\documentclass[11pt]{article}
\usepackage[utf8]{inputenc}
\usepackage{soul}
\usepackage{comment}
\usepackage{subcaption}
\usepackage{pgfplots}
\usepackage{tikz}

\newcommand{\required}[1]{\section*{\hfil \sharsokol2013advanced\hfil}}

\newcommand{\supp}{\operatorname{supp}}
\newcommand{\ssupt}{\operatorname{sing\, supp}}

\usepackage{verbatim,url}
\usepackage{graphicx, epsfig}
\graphicspath{{LRA_figs_1/}}
\usepackage{bbm}
\usepackage{amssymb,amsmath,amsfonts,amsthm}
\usepackage{upref}

\newcommand{\Beq}{\begin{equation}}
\newcommand{\Eeq}{\end{equation}}
\newcommand{\be}{\begin{equation}}
\newcommand{\ee}{\end{equation}}
\newcommand{\bs}{\begin{split}}
\newcommand{\beq}{\begin{equation*}}
\newcommand{\eeq}{\end{equation*}}
\newcommand{\bal}{\begin{align}}
\newcommand{\eal}{\end{align}}

\newcommand{\s}{\mathcal S}
\newtheorem{theorem}{Theorem}
\numberwithin{theorem}{section}
\newtheorem{proposition}[theorem]{Proposition}

\newcommand{\WF}{\mathrm{WF}}

\newcommand{\CH}{\mathcal H}
\newcommand{\hH}{\mathsf H}
\newcommand{\br}{\mathbb{R}}
\newcommand{\N}{\mathbb{N}}
\newcommand{\ik}{\varphi}

\newcommand{\bt}{\beta}
\newcommand{\al}{\alpha}
\newcommand{\la}{\lambda}

\newcommand{\ioi}{\int_0^{\infty}}
\newcommand{\ft}{f_T}

\newcommand{\dd}{\text{d}}

\newcommand{\ga}{\gamma}  
\newcommand{\chx}{\check {\bold x}}  
\newcommand{\chy}{\check {\bold y}}  
\newcommand{\chxc}{\check x}  
\newcommand{\chyc}{\check y}

\newcommand{\Hc}{\mathcal{H}}

\newcommand{\Eb}{\mathbb{E}}

\newcommand{\Pb}{\mathbb{P}}

\newcommand{\Rb}{\mathbb{R}}

\newcommand{\Nc}{\mathcal{N}}

\newcommand{\Ab}{\vec{\alpha}}

\newcommand{\e}{\epsilon}

\numberwithin{equation}{section}

\newcommand{\bel}[1]{\begin{equation}\label{#1}}

\newcommand{\eel}[1]{{\label{#1}\end{equation}}}

\newcommand{\paren}[1]{\left(#1\right)}               % Klammern
\renewcommand{\bar}[1]{\overline{#1}}

\newcommand{\vx}{\mathbf{x}}
\newcommand{\vy}{\mathbf{y}}

\newcommand{\cvx}{\check\vx}
\newcommand{\cvy}{\check\vy}
\newcommand{\Nrec}{N^{\text{rec}}}
\newtheorem{lemma}[theorem]{Lemma}

\newtheorem{definition}[theorem]{Definition}

\newtheorem{assumption}[theorem]{Assumption}

\theoremstyle{definition}

\newtheorem{remark}{Remark}

\usepackage{hyperref}
\hypersetup{
    colorlinks=true,
    linkcolor=blue,
    filecolor=magenta,      
    urlcolor=cyan,
}

\usepackage{xcolor}

\usepackage{authblk} 

\title{Statistical microlocal analysis in two-dimensional X-ray CT}
\author[$\dagger$]{Anuj Abhishek} 
\author[$*$]{Alexander Katsevich}
\author[$+$]{James W. Webber}

\affil[$\dagger$]{\small Department of Mathematics, Applied Mathematics and Statistics, Case Western Reserve University, USA\par axa1828@case.edu}

\affil[$*$]{\small Department of Mathematics, University of Central Florida, USA\par alexander.katsevich@ucf.edu}
\affil[$+$]{\small Department of Biomedical Engineering, Cleveland Clinic, USA\par webberj5@ccf.org }
\date{}
\begin{document}
\maketitle

\begin{abstract} 
In many imaging applications it is important to assess how well the edges of the original object, $f$, are resolved in an image, $f^\text{rec}$, reconstructed from the measured data, $g$. In this paper we consider the case of image reconstruction in 2D X-ray Computed Tomography (CT). Let $f$ be a function describing the object being scanned, and $g=Rf + \eta$ be the Radon transform data in $\br^2$ corrupted by noise, $\eta$, and sampled with step size $\sim\e$. Conventional microlocal analysis provides conditions for edge detectability based on the scanner geometry in the case of continuous, noiseless data (when $\eta = 0$), but does not account for noise and finite sampling step size. 
We develop a novel technique called \emph{Statistical Microlocal Analysis} (SMA), which uses a statistical hypothesis testing framework to determine if an image edge (singularity) of $f$ is detectable from $f^\text{rec}$, and we quantify edge detectability using the statistical power of the test. Our approach is based on the theory we developed in \cite{AKW2024_1}, which provides a characterization of $f^\text{rec}$ in local $O(\e)$-size neighborhoods when $\eta \neq 0$. We derive a statistical test for the presence and direction of an edge microlocally given the magnitude of $\eta$ and data sampling step size. Using the properties of the null distribution of the test, we quantify the uncertainty of the edge magnitude and direction. We validate our theory using simulations, which show strong agreement between our predictions and experimental observations. Our work is not only of practical value, but of theoretical value as well. SMA is a natural extension of classical microlocal analysis theory which accounts for practical measurement imperfections, such as noise and finite step size, at the highest possible resolution compatible with the data. %, and we hope this work will lead to further advancements in the field.
\end{abstract}

\section{Introduction}
In this paper, we introduce a new technique which quantifies the presence, direction and magnitude of an image edge when the image is reconstructed from finitely sampled and noisy tomographic data. We call this new technique ``Statistical Microlocal Analysis" (SMA). 
%The theory is an extension of two strands of Local Reconstruction Analysis (or, LRA), which is a recent technique developed by the authors to study images reconstructed from generalized Radon transform data. 

In recent work, the authors introduced Local Reconstruction Analysis (LRA), which studies images reconstructed from generalized Radon transform data in local neighborhoods. 
%The first strand is the analysis of resolution of reconstruction from discrete, noise free data, see, e.g., 
In \cite{Katsevich2017a,Katsevich2021a,Katsevich_2025_BV} (see references therein), we analyzed the resolution with which the singularities of $f$ can be reconstructed from discrete tomographic data in a deterministic setting, i.e., in the absence of noise. Later, in \cite{AKW2024_1, Katsevich_2025_BV}, this work was extended to address noise. We aim to combine this theory with statistical hypothesis testing to derive the foundations for SMA.

We now review some of the classical literature on microlocal analysis of generalized Radon transforms after establishing some standard notation that we will use for the review and throughout the paper. Let $g = \mathcal{R}f$ denote generalized Radon transform data in $\br^n$, where the transform, $\mathcal{R}$,  integrates $f$ over a family of $d < n$ dimensional (hyper)surfaces in $\br^n$. For simplicity, in this overview of existing literature we assume that $f$ is a conormal distribution. We assume that the \textit{singular support} of $f$, denoted $\ssupt(f)$, is a smooth, codimension one submanifold $\s \subset \mathbb{R}^n$. Consequently, $WF(f)\subset N^*\s$, where $WF(f)$ denotes the \textit{wavefront set} of $f$ and $N^*\s$ is the conormal bundle of $\s$. % If $f$ has a jump discontinuity across $\s$, then we say that $\s$ is an {\it edge} of $f$. With a slight abuse of notation, 
We refer to an element $(\vx,\xi)\in WF(f)$ as an ``edge" of $f$, where $\vx$ is the edge location and $\xi$ the direction.

%Classical microlocal analysis theory addresses the propagation of singularities from $f$ to the data $g$ and \tred{vice-versa.} Note that the reconstructed function, $f_r$, does not have to coincide with the original function $f$ since we generally consider parametrix-based \tred{(purely edge-preserving)} reconstruction, but we have $WF(f)=WF(f_r)$. In this case, detection of singularities of $f_r$ with continuous, noiseless data is based on the geometry of the \tred{integration surfaces} and on what data are available. 

Microlocal analysis techniques have been employed extensively in the imaging literature \cite{krishnan2015microlocal, WHQ, felea2013microlocal, webber2021microlocal, krishnan2012microlocal, grathwohl2020imaging, SU:SAR2013, Caday:SAR, ABKQ2013, Nguyen-Pham, Q1993sing}, and applied to a number of imaging fields, such as X-ray CT, Synthetic Aperture Radar (SAR), Compton scatter tomography, ultrasound and seismic imaging. In \cite{krishnan2015microlocal} (see references therein), a microlocal analysis of the classical hyperplane Radon transform, $R$, which has applications in X-ray CT, is overviewed. The authors state conditions for when an edge in $f$ is detectable in the reconstruction based on $g = Rf$ (the continuous, noiseless data).  
Specifically, let $(\vx, \xi) \in N^* \s$ be an edge to be reconstructed. If the data contains integrals over planes in a neighborhood of $L$, where $L$ is the tangent plane to $\s$ at $\vx$, then the edge is detectable (or ``visible" in the terminology of \cite{krishnan2015microlocal}). Otherwise, the edge is undetectable (or ``invisible"). Fourier Integral Operators (FIO) are one of the key tools in microlocal analysis for analyzing edge detection in image reconstruction. In \cite{krishnan2015microlocal}, the authors show that $R$ is an elliptic FIO which satisfies the Bolker condition \cite{GS1977, quinto}. This means that $\WF(R^*Rf) \subset \WF(f)$, i.e., the detected singularities are uniquely encoded in $g=Rf$.  The canonical relation relation of $R$ \cite[Definition 7]{krishnan2015microlocal} describes precisely how the singularities of $f$ propagate to $g$ and vice-versa. The authors go on to apply this theory to limited angle and exterior X-ray CT and provide simulated reconstructions to validate their results. Many other applications of microlocal analysis have been developed, too many to mention all of them here.

%\tred{Assuming that the Bolker condition is satisfied, we say that an edge is detectable} if there exists an algorithm, which is continuous between two Sobolev spaces, that recovers the edge. \tred{An edge is} undetectable means that the detection of the edge is exponentially unstable, that is, it is not continuous between any two Sobolev spaces of a finite order.

%\jc{Would suggest replacing this paragraph with the below.}
Assuming the Bolker condition is satisfied, if the edge is also visible, then one can often derive an algorithm to recover the edge which is continuous between two Sobolev spaces (a stable inverse). If the edge is invisible, then such an algorithm cannot exist.
%(do we want to add something like "without any further a-priori information about the scanned object?").}
% \jc{I would delete this last sentence as if Bolker doesn't hold then detectable as described above I'm not sure would imply this.}

%Suppose the continuous data are a set of integrals of $f$ over smooth surfaces $\s(\vy)$, $\vy \in Y$ (e.g., spheres, ellipsoids), and $Y$ parameterizes our data space. As mentioned above, an edge $(\vx,\xi) \in \WF(f)$ is detectable if there exists a $\vy\in Y$ such that $\vx\in\s(\vy)$ and $\xi$ normal to $T_\vx \s(\vy)$, the tangent plane to $\s(\vy)$ at $\vx$. If a $\vy\in Y$ with these properties does not exist, the edge at $(\vx,\xi) $ is invisible. This idea is applied extensively in much of the literature discussed above and is a generalization of the theory summarized in \cite{krishnan2015microlocal} (in this paper the $S(\vy)$ are planes). 

%The above result follows from the fact that $R$ is an elliptic Fourier Integral Operator (FIO) that satisfies the Bolker condition \cite{GS1977, quinto}. In this case, $R^*R$ is an elliptic Pseudo-Differential Operator ($\Psi$DO) and $\WF(R^*Rf) \subset \WF(f)$. In turn, this means that the visible singularities of $f$ are uniquely and stably encoded in $g=Rf$. 

The references listed above provide rigorous analyses of edge detectability in the case of continuous, noiseless integral transform data. To apply the theory in practice, we assume that the measured data is a reasonably accurate approximation to $\mathcal{R}f$. In cases, e.g., when the data is limited to continuous regions of sinogram space, such as limited angle tomography, this theory provides precise characterization of the image artifacts. %The theory then gives that if the idealized version of the collected data are sufficient to see all singularities of $f$, one should expect better image quality than if some data are missing and certain singularities are invisible. 

While such conclusions are valuable, the theory is only of limited practical use. The data are always discrete and noisy, but microlocal analysis does not yet address how such measurement errors affect edge detectability, which is one of the central questions in applications.

{\it The main contribution of this work is to extend classical microlocal analysis ideas to fully characterize how two of the main sources of error encountered in practice, namely noise and finite data sampling,  affect detection of singularities in the reconstructed image.} Moreover, our analysis applies at image scales of most interest in applications, namely at the scale of highest resolution consistent with the sampling step size and noise strength. If the data step size is $\sim\e$, for some $\e > 0$, and the noise is not too high magnitude, then our local reconstruction analysis is effective at the scale $\sim\e$ as well. This is the first ever such an extension. 

Our results are obtained using SMA, which employs statistical hypothesis testing. We claim that it is natural to use techniques from statistics because the data contains random noise, so the reconstructed image can be viewed as a random sample drawn from some statistical distribution. 

Let $f^{\text{rec}}_{\e,\eta}$ denote an image reconstructed from discrete, noisy classical Radon transform data in $\br^2$. The goal is to determine if a given $(\vx,\xi)\in T^*\mathbb{R}^2$ belongs to $\WF(f)$, and with what probability, using $f^{\text{rec}}_{\e,\eta}$. To this end, we define a random vector, $F$, which is calculated using appropriately weighted integrals of $f^{\text{rec}}_{\e,\eta}$ over a disk $D \subset \br^2$, $\text{diam}(D)=O(\e)$, centered at $\vx$. The theory of \cite{AKW2024_1} shows that the noise in $f^{\text{rec}}_{\e,\eta}$ is described locally by a continuous Gaussian Random Field (GRF). Furthermore, in \cite{Katsevich2021a} we derive the Discrete Transition Behavior (DTB) function, which describes how the edges of $f$ are smoothed in $f^{\text{rec}}_{\e,\eta}$ due to finite data sampling. By combining our GRF and DTB theories, we show that $F$ follows a bivariate Gaussian distribution and provide explicit expressions for its mean, $\mu$, and covariance, $\hat C$. These depend on the data sampling step size, $\e$, the noise level, $\sigma^2$, and the jump (edge) magnitude $\Delta f$. The null hypothesis of our statistical test is that there is no edge present. In this case, $\mu = (0,0)$, and we establish conditions based on $\e$ and the ratio $\sigma/|\Delta f|$ when to reject the null hypothesis (e.g., using a $95\%$ confidence interval). We also calculate explicitly the statistical power, $1-\beta$, of the test, which represents the probability that a true edge is detected correctly. 

Our approach establishes an important parallel with classical microlocal analysis theory. In works such as \cite{krishnan2015microlocal}, the detectability of an edge is based on data availability (i.e., which continuous regions of sinogram space are accessible), and the result is deterministic. In contrast, we provide a probability for edge detection microlocally based on $\Delta f$, $\e$, and $\sigma^2$. To adopt the terminology of \cite{krishnan2015microlocal} and much of the classical literature, we say {\it the singularity is ``visible" with probability $1-\beta$}. 
%We note also that our analysis is based on the hyperplane Radon transform which satisfies the Bolker condition \cite{krishnan2015microlocal}, and thus we do not need to consider additional, unwanted edges (artifacts) occuring in the reconstruction due to violation of Bolker, and we focus purely on edge visibility.

To continue the analogy with classical microlocal analysis, we view $F$ as an estimate of a conormal vector, which encodes the magnitude and direction of the edge. We show that when the noise is zero, one has $F=H$, where $|H|$ is proportional to $\Delta f$ (up to a known constant) and $H$ is normal to the edge. When the data are noisy, the estimate $F$ deviates from $H$. 

In addition to locating the edge, we also provide a {\it confidence region for the vector $H$} at any prescribed confidence level. It is an ellipse centered at $F$. We then use the probability density function (PDF) of $F$ derived from our theory to quantify how spread out the direction, $\Theta:=F/|F|\in S^1$, and magnitude of the edge, $|F|$, are from the true values, $\Theta_0:=H/|H|$ and $|H|$, respectively, when computed from discrete noisy data. 

%The strength of singularities is discussed in the microlocal analysis literature, e.g., in \cite{felea2013microlocal} \tred{a different reference?} using $I^p$ classes, and on Sobolev scale -- in \cite{Q1993sing}. While this accounts for the order of singularities (e.g., jump singularities) it does not address edge magnitude due to scaling. For example, it is known that the characteristic function $\chi_B$, where $B$ is the unit ball in $\mathbb{R}^n$, is in $H^{\alpha}(\mathbb{R}^n)$ for any $\alpha < 1/2$, where $H^{\alpha}$ denotes the Sobolev space order $\alpha$ \cite[section IV.2 and section VII]{natterer}. Also, $\chi_B\in I^{-1}(N^*\pa B)$. Both of these descriptions are invariant under scaling and thus $\alpha$ cannot be used to quantify the jump magnitude. \tred{The notion of the principal symbol does quantify the jump magnitude. We may need to rethink this paragraph or drop it altogether.} We address this here using $|F|$.

To validate our theory, we conduct simulated experiments where we reconstruct the characteristic function of a ball from line integral data. Pseudorandom noise is added using draws from a uniform distribution. The observed distribution for $F$, which we estimate using histograms, accurately matches with our predictions, and $\beta$ is estimated accurately. We plot confidence bands for $F$ and quantify the spread for the edge magnitude and direction. In addition, we include several visualizations to help illustrate our method. For example, we plot polar graphs of the edge direction likelihood, to help the reader visualize the directions in which the edge is most likely to occur. 
%In addition, we use our method to recover the wavefront set of the reconstructed image. We do this by translating the integral domain, $D$, over the scanning region to identify the likely edge locations and directions. 

Our results are also of significant practical importance. 
%The main practical application of our theory is to identify whether an edge of $f$ at $\vx \in \s$ is visible (in the new sense) based on $\e$ and $\sigma/|\Delta f|$. This provides important information to practitioners as to which, e.g., $\Delta f$ values they are able to resolve in the image and with which probability based on the data sampling step size and noise level. 
They provide practitioners with a tool to assess how well and with what probability the edges are resolved in the reconstructed image based on the noise level, data sampling rate, and edge magnitude (or, jump size). {Our theory is sensitive to the spatial variability of noise in the data. For example, if there is higher noise variance corresponding to rays which intersect a given domain, $D$, of the image being assessed, then this affects edge resolvability in $D$.} In medical applications, this allows for quantification of how likely it is to detect (or miss) a diagnostically valuable feature {at a particular location} in an image (e.g., a small tumor in $D$). Likewise, in nondestructive testing, our method informs the scanner operator of how likely it is to identify a feature that is critical for the structural integrity of a scanned part. Numerous other examples are possible.

The work in this paper is closely related to the problem of edge detection, which arises in image analysis and statistics. Edge detection in imaging applications has a rich history, \cite{canny86}. In particular, in statistical image reconstruction, boundary detection is often formulated as a spatial change point problem, where points on the boundary (edges) are regarded as change points. In \cite{Tsyb_92_cp,minimaxIR_book} the authors provide an analysis of multidimensional change-point problems and boundary estimation and prove global minimax-optimality for detecting such spatial discontinuities under the assumption of Gaussian noise. In \cite{hou03,hoon06} edge detection algorithms are proposed, which are robust to noise. Most of these works study direct image edge estimation problems. In this setting, the image edge is an ideal, sharp edge, and noise in the image is uncorrelated. 

In \cite{tsyb_06,tsyb_08} change-point estimation from indirect observations in $\br$ was considered for convolu\-tion-type indirect data. However, the model and the results in these works are very different from ours. There, the noise is the standard two-sided Wiener process on $\br$ with magnitude $O(\e)$, $\e\to0$, while the nonrandom part of the signal (the edge) is convolved with a {\it fixed} kernel (independent of $\e$). Hence, in principle, it is not difficult to separate high frequency noise from low frequency nonrandom signal (especially as $\e\to0$). Our work differs from the above in that (1) both the support of the convolution kernel and the noise magnitude depend on $\e$, and (2) we work at the native length scale $\sim\e$. This makes our task very different, because at this scale {\it noise in the reconstruction is a smooth random function} and it is harder to distinguish it from an equally smooth nonrandom signal. Taking the limit as $\e \to0$ does not help to separate the two, because they both remain smooth in our scaling regime. Finally, we work in $\br^2$ and estimate both the magnitude and the direction of the edge.

A recent, very interesting method to find object boundaries and quantify their roughness from indirect measurements using Bayesian methods is in \cite{afkham_24_uq}. However, the goals of the study and method of analysis are very different from ours. 

We emphasize that SMA and the results presented here are the first steps towards new theory, and there remain many open questions and interesting avenues to explore.  

The remainder of the paper is organized as follows. In Section \ref{sec:setting_mainres}, we establish our notation and assumptions (e.g., on noise) and review key results from previous work which will be needed to prove our theorems. In sections  \ref{SMA} and \ref{sec:anal2D}, we introduce SMA and prove our main theoretical results. In Section~\ref{SMA}, a simpler 1D setting is considered. The assumption here is that the direction normal to the edge is known. This case serves primarily as an illustration of our main ideas. The most practically relevant 2D case is considered in Section~\ref{sec:anal2D}, where the goal is to determine whether $\vx\in\s$ and, if so, estimate the magnitude of the jump and the normal direction to $\s$ at $\vx$. Here, we develop a pointwise statistical test for the presence of the edge and provide expressions for the power, $1-\beta$. We present simulated experiments in Section \ref{experiments} to validate our theory. The 1D case is illustrated in Section~\ref{ssec:1dexp}, and the 2D case - in Section~\ref{ssec:2dexp}. In sections~\ref{ssec:edm}--\ref{ssec:macro} we investigate additional applications of SMA in 2D, such as describing the uncertainty in the edge direction and magnitude. We also apply our 2D test in a scanning regime, where we apply the pointwise test, designed for local domains, across a full 2D image. This is only for illustration purposes since we have not yet expanded the statistical guarantees of our pointwise test to the full image domain. To derive such results, one would have to account for correlation between neighboring image patches, and this task is beyond the scope of the paper. Finally, an auxiliary proposition is stated and proven in the appendix.

\section{Setting of the problem.}\label{sec:setting_mainres}
We consider the problem of reconstructing a function $f(\vx)$, $\vx\in \Rb^2$, from discretely sampled noisy Radon Transform  (RT) data. We use the following parameters to  discretize the observation (or, data) space, $S^1\times [-P,P]$:
\be\label{params}
\al_k=k\Delta\al\in[-\pi,\pi),\ p_j=\bar p+j\Delta p\in[-P,P],\ \Delta p=\e,\ \Delta\al/\Delta p=\kappa,
\ee 
where $\kappa>0$ and $\bar p$ are fixed. Here $p$ is the affine parameter, and $\al\in[-\pi,\pi)$ corresponds to $\vec\al=(\cos\al,\sin\al)\in S^1$. We loosely refer to $\epsilon$ as the data step size. The discrete noisy tomographic data are modeled as:
\be\label{disc_model}
\hat{f}_{\e,\eta}(\al_k,p_j)= Rf(\al_k,p_j)+ \eta_{k,j},
\ee
where $Rf(\al_k,p_j)$ is the Radon transform of $f$ at the grid point $(\al_k,p_j)$, and $\eta_{k,j}:=\eta(\al_k,p_j)$ are random variables that model noise in the observed data. We assume $\eta_{k,j}$ are zero mean, independent but not necessarily identically distributed. 

We now state our main assumptions on $f$ and the measurement noise. Given a set $D\subset\br^2$, let $\chi_D$ denote the characteristic function of $D$.
\begin{assumption} {(Assumptions on $f$)}\label{ass:f}
\begin{enumerate}
\item $\supp f\subset\{\vx\in\br^2:\ |\vx|<P\}$.
\item\label{pwsm} $f(\vx)=\sum_{k=1}^K f_k(\vx)\chi_{D_k}(\vx)$ for some $K\in\N$, $f_k\in C^\infty(\br^2)$ and some open sets $D_k\subset\br^2$ with piecewise smooth boundary.
\end{enumerate}
\end{assumption}
By assumption~\ref{ass:f}\eqref{pwsm}, we limit our discussion to functions $f$ that, at worst, have jump type singularities. 

%In Assumption \ref{noi}, we state our assumptions on the first three moments of the random variables $\eta_{k,j}$.
%\noindent

For convenience, throughout the paper, we use the following convention: if a constant $c$ is used in an equation or inequality, the qualifier ‘for some $c>0$’ is assumed. If several $c$’s are used in a string of (in)equalities, then ‘for some’ applies to each of them, and the values of different $c$’s may all be different. For example, in the string of inequalities $f\le cg \le ch$, the values of $c>0$ in two places may be different.

Let $\Eb(X)$ denote the expected value of a random variable $X$. 

\begin{assumption} {(Assumptions on noise)}\label{noi}
\begin{enumerate}
\item $\Eb(\eta_{k,j})=0$.
\item\label{sig_2} $\Eb\eta_{k,j}^2=\sigma^2(\al_k,p_j)\Delta\al$ for some even $\sigma\in C^1([-\pi,\pi]\times[-P,P])$, i.e. $\sigma(\al,p)=\sigma(\al+\pi,-p)$.
\item $\Eb \lvert \eta_{k,j}\rvert^3\leq c(\Delta \alpha)^{3/2}$.
\end{enumerate}
\end{assumption}

See a discussion at the end of this section about how these conditions can be relaxed.
For reconstruction, we use an interpolation kernel, $\ik$, which satisfies the following assumptions.
\begin{assumption}{(Assumptions on the kernel $\ik$)}\label{interp}
\begin{enumerate}
\item $\ik$ is compactly supported.
\item $\ik^{(M+1)}\in L^\infty(\br)$ for some $M\ge 3$.
\item $\ik$ is exact up to order 1, i.e. $\sum_k \ik(t-k) k^j\equiv t^j$, $j=0,1$. 
\end{enumerate}
\end{assumption}
Note that the last assumption implies $\int \ik(t)dt=1$.
%Usually we make an additional assumption that $\ik$ exactly interpolates polynomials up to some degree $m_{\max}$ \cite{Katsevich2017a, kat19a, Katsevich2020a, Katsevich2020b, Katsevich2021a}: 
%\be\label{interp cond}
%\sum_{j\in \Zb}j^m\ik(t-j)=t^m,\ t\in \Rb,\ m=0,1,\dots,m_{\max}.
%\ee
%Here this assumption is not needed. In particular, $\ik$ may account for the effects of smoothing that can be used to reduce noise in the reconstruction. In this case, $\ik$ no longer satisfies \eqref{interp cond}.
Denoting the Hilbert transform by $\Hc$, the reconstruction formula from the data \eqref{disc_model} is given by:
\begin{align}\label{recon-orig}
f^{\text{rec}}_{\e,\eta}(\vx)&=-\frac{\Delta\al}{4\pi\e}\sideset{}{_{|\al_k|\le \pi}}\sum \sideset{}{_j}\sum\CH \ik^{\prime}\left(\frac{\vec\al_k\cdot \vx-p_j}\e\right)\hat f_{\e,\eta}(\al_k,p_j)= f^{\text{rec}}_{\e}(\vx)+N^{\text{rec}}_{\e}(\vx),
\end{align}
where 
\be\label{fN recons}\begin{split}
   f^{\text{rec}}_{\e}(\vx)&  :=-\frac{\Delta\al}{4\pi\e}\sideset{}{_{|\al_k|\le \pi}}\sum \sideset{}{_j}\sum\CH \ik^{\prime}\left(\frac{\vec\al_k\cdot \vx-p_j}\e\right) Rf(\al_k,p_j),\\
   N^{\text{rec}}_{\e}(\vx)&:=-\frac{\Delta\al}{4\pi\e}\sideset{}{_{|\al_k|\le \pi}}\sum \sideset{}{_j}\sum\CH \ik^{\prime}\left(\frac{\vec\al_k\cdot \vx-p_j}\e\right)\eta_{k,j}.
\end{split}
\ee 
Since $\ik$ is compactly supported, we see from \eqref{fN recons} that the resolution of the reconstruction is, roughly, of order $\sim\e$, i.e. of the same order as the data step size. 

Pick a point $\vx_0\in \s:=\ssupt f$. Let $\vec\Theta_0$ be a unit vector normal to $\s$ at $x_0$. Now we state a key technical assumption on $\vx_0$ that is needed to use our previous results. Let $\langle r\rangle$ denote the distance from a real number $r\in\br$ to the integers, $\langle r\rangle:=\text{dist}(r,\mathbb Z)$. The following definition is in \cite[p. 121]{KN_06} (after a slight modification in the spirit of \cite[p. 172]{Naito2004}).

\begin{definition} Let $\nu>0$. The irrational number $s$ is said to be of type $\nu$ if $\nu$ is the infimum of all $\nu_1$ for which there exists $c(s,\nu_1)>0$ such that
\be\label{type ineq}
m^{\nu_1}\langle ms\rangle \ge c(s,\nu_1) \text{ for any } m\in\mathbb N.
\ee
\end{definition}
See also \cite{Naito2004}, where the numbers which satisfy \eqref{type ineq} are called $(\nu-1)$-order Roth numbers. It is known that $\nu\ge1$ for any irrational $s$.

\begin{assumption}{(Assumptions on $\vx_0$)}\label{ass:x0}
\begin{enumerate}
\item\label{x0_1} The quantity $\kappa|\vx_0|$ is irrational and of some finite type $\nu$.
\item\label{x0_2} $\sigma^2(\al,\vec\al\cdot \vx_0)\not=0$ for all $\al$ in some open set $\Omega\subset[0,2\pi]$.
\item\label{x0_3} $|\vx_0|<P$.
\item\label{x0_4} The curvature of $\s$ at $\vx_0$ is not zero.
\item\label{x0_5} The quantity $\kappa \vx_0\cdot\vec\Theta_0^\perp$ is irrational, where $\vec\Theta_0\in S^1$ is normal to $\s$ at $\vx_0$.
\item\label{x0_6} The line $\{\vx\in\br^2:(\vx-\vx_0)\cdot\vec\Theta_0=0\}$ is not tangent to $\s$ at any point where the curvature of $\s$ is zero.
\end{enumerate}
\end{assumption}

The asymptotic behaviour of $f^{\text{rec}}_{\e}(\vx)$ in a small neighborhood of $\vx_0$ as $\e \to 0$, is well understood from the theory of local reconstruction analysis (LRA), see e.g. \cite{Katsevich2017a, kat19a, Katsevich2020a, Katsevich2020b, Katsevich2021a}. Let 
\be
\Delta f(\vx_0):=\lim_{t\to0^+} \big(f(\vx_0+t\vec\Theta_0)-f(\vx_0-t\vec\Theta_0)\big)
\ee
be the value of the jump of $f$ at $\vx_0$. Let $D\subset\br^2$ be a bounded domain. It is proven in \cite{Katsevich2017a, Katsevich2021a} that
\be\label{DTB appr}
f^{\text{rec}}_{\e}(\vx_0+\e\cvx)=c(\vx_0,\e)+\Delta f(\vx_0)\ft(\vec\Theta_0\cdot\cvx)+O(\e),\ \cvx\in D,\ \e\to0.
\ee
Here $\cvx$ is a point in the rescaled and shifted coordinates, $\cvx:=(\vx-\vx_0)/\e$, and
\begin{equation}\label{fr}
\begin{split}
\ft(t) = \int_{-\infty}^t \ik(s)\dd s-\frac12
\end{split}
\end{equation}
is the DTB function. Since $\int\ik(t)\dd t=1$, we have $\lim_{t\to\pm\infty}f(t)=\pm1/2$. Thus, $f_T$ is a smoothed edge of unit magnitude centered at zero. 

Next we cite the main result of \cite{AKW2024_1}, Theorem 2.10, which describes the behavior of $N_\e^{\text{rec}}$ as $\e\to0$. Given two real-valued functions $f$ and $g$, their cross-correlation is defined as follows:
\be\label{cross-cor}
(f\star g)(t):=\int_\br f(t+s)g(s)\dd s
\ee
as long as the integral is well-defined.

\begin{theorem}\label{GRF_thm}
Let $D$ be a rectangle. Suppose assumptions~\ref{noi}, \ref{interp}, and \ref{ass:x0} are satisfied and $M>\max(\nu+1,3)$. Then, $N_{\e}^{\text{rec}}(\vx_0+\e \chx)\to N^{\text{rec}}(\chx)$, $\chx\in D$, $\e\to0$, as GRFs in the sense of weak convergence. Furthermore, $N^{\text{rec}}(\chx)$ is a GRF with zero mean and covariance $\text{Cov}(\chx,\chy)=C(\chx-\chy)$, where
\be\label{Cov main}
C(\chx)=\bigg(\frac{\kappa}{4\pi}\bigg)^2\int_{0}^{2\pi}\sigma^{2}(\alpha,\Ab\cdot \vx_0) (\varphi^{\prime}\star \varphi^{\prime})(\Ab\cdot \chx)\dd\alpha,
\ee
and the sample paths of $N^{\text{rec}}(\chx)$ are continuous with probability $1$.
\end{theorem}

It follows from assumption~\ref{noi}\eqref{sig_2} that $C$ is even: $C(-\chx)=C(\chx)$, $\chx\in\br^2$.
Assumptions~\ref{ass:x0}(\ref{x0_1}--\ref{x0_3}) are needed for Theorem~\ref{GRF_thm} to hold, and assumptions~\ref{ass:x0}(\ref{x0_3}--\ref{x0_6}) -- for \eqref{DTB appr} to hold.

In the rest of the paper we assume $\e\ll 1$ is sufficiently small and the reconstruction from nosiy data, $\hat f_{\e,\eta}$, can be accurately approximated using \eqref{DTB appr} and Theorem~\ref{GRF_thm}:
\be\label{complete appr}
f_{\e,\eta}^{\text{rec}}(\vx_0+\e\cvx)\approx c(\vx_0,\e)+\Delta f(\vx_0)\ft(\vec\Theta_0\cdot\cvx)+N^{\text{rec}}(\chx),\ \cvx\in D,\ \e\ll1.
\ee
Due to the linearity of the Radon transform inversion \eqref{recon-orig}, we will assume without loss of generality that $\Delta f(\vx_0)=1$ in what follows. In this case, $\sigma$ represents the noise standard deviation relative to the magnitude of the jump, i.e. the relative noise strength. 

Assumption~\ref{noi} is not restrictive and can be significantly relaxed \cite{AKW2024_1}. Suppose the second and third moments of $\eta_{k,j}$ satisfy
\be\label{flex moments}
\Eb\eta_{k,j}^2=\sigma^2(\al_k,p_j)\Delta\al\vartheta^2(\e),\ 
\Eb \lvert \eta_{k,j}\rvert^3\leq c(\Delta \alpha\vartheta^2(\e))^{3/2},
\ee
where $\vartheta$ is any nonzero function. Theorem~\ref{GRF_thm} still holds provided it is applied to $N_\e^{\text{rec}}/\vartheta(\e)$ and $N^{\text{rec}}/\vartheta(\e)$ \cite{AKW2024_1}.

We present assumption~\ref{noi} as above because in this case the level of noise in the data is such that $N^{\text{rec}}$ (the noise component in $f_{\e,\eta}^{\text{rec}}$) is comparable with the deterministic, useful part of the reconstruction, $\Delta f(\vx_0)\ft$ (cf. \eqref{complete appr}). 

Suppose now $\vartheta(\e)\not\sim 1$. If $\vartheta(\e)\to0$, then $N^{\text{rec}}\to0$, and the noise vanishes in the limit. This makes the task of edge detection trivial. If $\vartheta(\e)\to\infty$, then $N^{\text{rec}}\to\infty$, and the noise overwhelms the deterministic part of the reconstruction. This means that $\e$ is no longer {an appropriate} resolution scale. 

To estimate what the appropriate resolution should be, replace the data \eqref{params}, \eqref{disc_model} with
\be\label{params pr}
\begin{split}
&\al_k=k\Delta\al^{(1)},\ p_j=\bar p+j\Delta p^{(1)},\ \Delta p^{(1)}=N\Delta p=N\e,\ \Delta\al^{(1)}=N\Delta \al,\\
&\hat{f}_{\e,\eta}^{(1)}(\al_k,p_j)= \frac1{N^2} \sum_{-N/2\le m,n< N/2}\hat{f}_{\e,\eta}(\al_{k+m},p_{j+n}),\ k,j\in N\mathbb Z.
\end{split}
\ee
This is the common procedure of pixel binning, whereby $N\times N$ neighboring measurements are averaged into one. Clearly, this increases the native scale from $\e$ to $\e^\prime=N\e$ and reduces the standard deviation of noise  by a factor of $\sim N$. Here we use that $\sigma\in C^1$, and the standard deviations of the $\eta_{k,j}$ that are averaged into one random variable are approximately the same. 

The new pixels do not overlap, therefore the random errors in the new data are still uncorrelated. By choosing $N\gg1$ (as a function of $\e$) we can find $\e^\prime$ such that the new (averaged) noise $\eta_{k,j}^\prime$ satisfies Assumption~\ref{noi} with $\e^\prime$ instead of $\e$. Once this is achieved, the value of $\e^\prime$ is the appropriate resolution scale.

In the above, we tacitly assume that $\e^\prime\ll1$. If this is not the case, then noise in the data is too strong and formula-based reconstruction does not give reliable results. 

Note that the restriction on $\Eb \lvert \eta_{k,j}\rvert^3$ in \eqref{flex moments} is quite natural; it follows from the inequality $\Eb \lvert \eta_{k,j}\rvert^3\le c (\Eb\eta_{k,j}^2)^{3/2}$. The latter inequality holds for many random variables, which include uniform and Gaussian random variables and, more generally, those with logconcave density (see e.g. \cite[Theorem 5.22]{LV2007}, \cite[Lemma 24]{vemp24}).

\section{Statistical Microlocal Analysis}\label{SMA}
In this section we develop a statistical test to determine whether $f$ has a jump at a fixed point $\vx_0\in\s$ given an image $f_{\e,\eta}^{\text{rec}}$ reconstructed from noisy discrete Radon transform data $\hat{f}_{\e,\eta}$ in \eqref{disc_model}. 
%We also quantify the uncertainty in the direction, $\vec\Theta_0$, and value, $\Delta f(\vx_0)$, of the posited jump based on our statistical framework. We call our approach Statistical Microlocal Analysis. To illustrate our ideas, we first begin with the case of singularity detection along a fixed line. 
We use the standard hypothesis testing framework, \cite[Section 8.2.1]{CB_book}, \cite[Section 7.1]{Hardle}, \cite[Section 14.4.4]{LeRo22}, \cite[Section 9.2]{Fieguth22}. In the following subsection, we briefly recall the ideas relevant to our analysis largely following these four references. 

\subsection{Primer on statistical hypothesis testing}\label{ssec:primer}
We consider a $d$-dimensional Gaussian random vector $F \sim \mathcal{N}(\mu, \hat{C})$. Suppose $\mu=0$ under the null hypothesis $\hH_0$, and $\mu = H \neq 0$ under the alternative $\hH_1$:
\begin{align}\label{hyp1}
   \hH_0: F \sim \mathcal{N}(0, \hat{C}); \quad \hH_1: F \sim \mathcal{N}(H, \hat{C}),H \neq 0.
\end{align}
We assume throughout that the covariance matrix $\hat{C}$ is known. 
%We use the following Likelihood Ratio Test (LRT), which compares the likelihoods of $F$ under $\hH_1$ and under $\hH_0$. 
Recall that the PDF of a multivariate normal distribution with mean $\mu$ and covariance matrix $\hat{C}$ is given by:
\begin{equation}\label{main pdf}
 f(F \mid \mu, \hat{C}) = \big[(2\pi)^d |\hat{C}|\big]^{-1/2} \exp\left( -\Vert F - \mu\Vert_{\hat{C}^{-1}}^2/2 \right),\ \Vert F\Vert_{\hat{C}^{-1}}:= \big(F^T \hat{C}^{-1} F\big)^{1/2}.
\end{equation}
Here $|\hat{C}|$ is the determinant of $\hat C$. 
%In our case, the \textit{likelihoods} under $\hH_0$ and $\hH_1$ are given by 
%\begin{align}\label{l0}
%   &L_0(F):= f(F |\hH_0) = \big[(2\pi)^d |\hat{C}|\big]^{-1/2} \exp\left( -\Vert F \Vert_{\hat{C}^{-1}}^2/2 \right),\\
%\label{l1}
%   &L_1(F;H) :=f(F | \hH_1) = \big[(2\pi)^d |\hat{C}|\big]^{-1/2} \exp\left( -\Vert F - H\Vert_{\hat{C}^{-1}}^2/2 \right),
%\end{align}
%respectively. 
According to the Likelihood Ratio Test (LRT) theory, given the observed vector $F$, the test statistic is computed based on the following ratio
\begin{equation}
    \Lambda(F) = \log \frac{\max_{H\in\br^2}f(F;H)}{f(F;0)}.
\end{equation}
A higher value of the log-likelihood ratio indicates that the observation is more likely to be generated by the alternative rather than the null hypothesis. The maximum in the numerator is attained when \( H=F \) and the corresponding value is $\big[(2\pi)^d |\hat{C}|\big]^{-1/2}$. Simplifying and multiplying by 2 gives the test statistic used in this work:
\be\label{Z primer}
Z:=2\Lambda(F) = \Vert F\Vert_{\hat{C}^{-1}}^2 .
\ee

It is well-known that under $\hH_0$ the test statistic follows a $\chi$-squared distribution with $d$ degrees of freedom: $Z \sim \chi^2_d$. The LRT of size $\al$ (i.e., with the probability of type I error equal to $\al$) for testing $\hH_0$ against $\hH_1$ rejects \( \hH_0 \) if $Z > c_{\alpha}$, where $c_{\alpha}$ solves 
\be\label{level}
P(Z> c_{\alpha} \mid \hH_0) = P(\chi_d^2 > c_{\alpha})=\alpha. 
\ee
If the cumulative distribution function (CDF) of the $\chi_d^2$ distribution is denoted $\Upsilon_0$, then $c_{\alpha}=\Upsilon_0^{-1}(1-\alpha)$. 

We are also interested in calculating the \textit{power} of the statistical test, i.e. the probability of correctly rejecting \( \hH_0 \) when \( \hH_1 \) holds with some fixed mean vector $H$. For a chosen tolerance level $\alpha$ (type I error), if the null distribution is $\Upsilon_0$ and the distribution of the alternative is $\Upsilon_1$, the probability of making a type II error $\beta$ is: 
\be
\beta=1-\Upsilon_1(\Upsilon_0^{-1}(1-\alpha)).
\ee
The power of the test is then $1-\beta$.

In Section \ref{experiments} below, we use the Receiver Operating Characteristic (ROC) curve to demonstrate the effectiveness of the proposed tests. The ROC curve plots the True Positive Rate (TPR) or statistical power ($1-\beta$) against the False Positive Rate (FPR) ($\alpha$).

In the curve in Figure~\ref{fig:roc}, the $x$-axis represents $\alpha$
and the $y$-axis represents $1-\beta$. Furthermore, each point on the curve corresponds to a different decision threshold. We note that increasing $\alpha$ increases $1-\beta$ as well (i.e. more true positives but more false positives).
On the other hand, reducing $\alpha$ increases $\beta$ (i.e. fewer false positives but more false negatives). Ideally, one would like to maximize  $1 - \beta $ while keeping $\alpha$ low. For illustration,
an example ROC curve is plotted in Figure~\ref{fig:roc}, where the diagonal line represents a random classifier. 

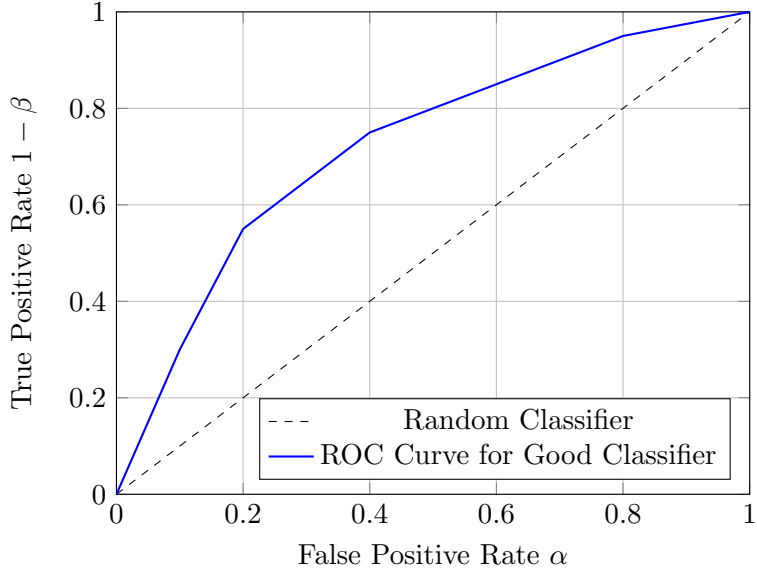
\begin{figure}
\begin{center}
\begin{tikzpicture}
    \begin{axis}[
        xlabel={False Positive Rate \( \alpha \)},
        ylabel={True Positive Rate \( 1 - \beta \)},
        xmin=0, xmax=1,
        ymin=0, ymax=1,
        grid=major,
        width=10cm, height=8cm,
        legend pos=south east
    ]
        % Random classifier (diagonal line)
        \addplot[color=black, dashed] coordinates {(0,0) (1,1)};
        \addlegendentry{Random Classifier}

        % Example ROC curve
        \addplot[color=blue, thick] coordinates {(0,0) (0.1,0.3) (0.2,0.55) (0.4,0.75) (0.6,0.85) (0.8,0.95) (1,1)};
        \addlegendentry{ROC Curve for Good Classifier}
    \end{axis}
\end{tikzpicture}
\end{center}
\caption{Illustration of the ROC curve.}
\label{fig:roc}
\end{figure}

Lastly, \textit{Area Under the ROC Curve} (AUC) measures the classifier's overall performance. For a random classifier, AUC $=0.5$ (for the dashed line) and the closer the AUC is to $1$, the better the classifier is. 

\subsection{Edge detection in 1D}\label{first_prin}
In this section, we describe a statistical test for detecting edges along a fixed direction. 
We assume that if the edge is present, its direction $\vec\Theta_0$ is known. The goal of this section is to illustrate the main ideas of SMA in a simpler setting. 

We begin by considering weighted integrals of the reconstructed image over a line segment. 
Select an odd, compactly supported function $u\in L^\infty(\br)$ and consider the quantity
\begin{equation}\label{Fu}
\begin{split}
F_u &= \int_{\br} u(t)\ft(t)\dd t + \int_{\br} u(t)N^{\text{rec}}(t)\dd t
=: H_u + G_u,
\end{split}
\end{equation}
where $N^{\text{rec}}$ is the part of the reconstruction only from noise and $f_T$ is the deterministic part of $f_{\e,\eta}^{\text{rec}}$ (the DTB function), see \eqref{fr} and \eqref{complete appr}. Since $u$ is odd, the constant term in \eqref{complete appr} integrates to zero. With a slight abuse of notation, we denote $N^{\text{rec}}(t):=N^{\text{rec}}(t\vec\Theta_0)$.
Thus, $F_u$ represents a weighted integral of the reconstruction over the line segment $\{\vx_0 + \e t\vec\Theta_0 : t\in\supp u\}$ in physical coordinates.  

We assume $\supp u \subset [-\rho,\rho]=:I_\rho$, where $I_\rho$ is the edge detection window. Note that $\rho$ is in units of $\e$, so $\rho = 1$ implies $I_\rho$ has width $2\e$ in physical coordinates.
%In regards to the statistical test we develop later in this section, a smaller $\rho$ means we can we can detect the edge more accurately (i.e., at a higher resolution), and vice-versa. 

Our test statistic is based on the random variable $F_u$, so we examine the distribution function of $F_u$ in two cases: 
\be\label{two hypos}
\hH_0:\ \vx_0 \notin \s \text{ (no edge is present}),\quad  \hH_1:\ \vx_0 \in \s \text{ (edge is present}).
\ee
Suppose first $\vx_0\not\in\s$. Then $H_u=0$ and $F_u=G_u$. By Theorem~\ref{GRF_thm}, $N^{\text{rec}}(t)$, $t\in I_\rho$, is a continuous Gaussian random function with zero mean and covariance $C_1(t-s)$, where 
\be\label{C1 def}
C_1(t) = \paren{\frac{\kappa}{4\pi}}^2 \int_0^{2\pi}\sigma^2(\al,\vec\al\cdot\vx_0)(\varphi^{\prime}\star \varphi^{\prime})\big(\cos(\al-\theta_0)t\big)\dd \al.
\ee
It follows from Proposition \ref{prop2} that $G_u$ is a Gaussian random variable, $G_u\sim \mathcal{N}(0,\gamma^2)$,
where
\begin{equation}\label{ga sq}
\gamma^2 = \int_{\br}\int_{\br}u(t_1)u(t_2)C_1(t_1-t_2)\dd t_1\dd t_2.
\end{equation}
In a similar fashion, $F_u \sim \mathcal{N}(H_u,\gamma^2)$ if $\vx_0\in\s$. 

The above derivation shows that we can apply the LRT theory outlined in Section~\ref{ssec:primer} with $d=1$. The statistic to test $\hH_0$ against $\hH_1$ is given by $Z = (F_u/\gamma)^2$. As is well-known, the CDF of $\chi_1^2$ (i.e., of the $\chi_d^2$ distribution with $d=1$ degree of freedom) is 
\be
\Upsilon_0(x)=\text{erf}(\sqrt{x/2}),\ 
\text{erf}(x) = \frac{2}{\sqrt{\pi}} \int_0^x e^{-t^2} dt.
\ee
Set $c_\al=\Upsilon_0^{-1}(1-\alpha)$. Hence $\hH_0$ is rejected if $(F_u/\ga)^2>c_{\alpha}\iff|F_u|/\ga>\sqrt{c_\al}$. Using that $F_u/\ga \sim \mathcal{N}(H_u/\ga,1)$ under $\hH_1$, we immediately obtain the probability of type II error:
\be\label{beta 1d}
\beta = (1/2)\big[\text{erf}(|H_u|/\gamma+c_\al) -\text{erf}(|H_u|/\gamma-c_\al)\big].
\ee
Our arguments prove the following result.

\begin{theorem}[1D edge detection]\label{th3.1}
Let $F_u$ be the random vector given by eq. \eqref{Fu}. Set $Z = (F_u/\gamma)^2$. The LRT of size $\al$ for testing $\hH_0$ against $\hH_1$, which are defined in \eqref{two hypos} (no edge vs. edge), rejects $\hH_0$ if $Z > \Upsilon_0^{-1}(1-\alpha)$. The power of the LRT is $1-\bt$, where $\bt$ is given in \eqref{beta 1d}.
\end{theorem}

Since $\ga$ is proportional to the relative strength of the noise (noise-to-signal ratio, NSR), \eqref{beta 1d} implies that, as expected, the power of detecting an edge grows as $\ga$ decreases (SNR increases). 

So far we have considered $\rho$ to be fixed. It is reasonable to expect that as the edge detection window $I_\rho$, increases, the test becomes more sensitive. Hence we wish to analyze the behaviour of $\beta$ when $\rho$ increases (i.e., as we use more of the image around the edge). 

Suppose $u(t)=\chi_\rho(t)\text{sgn}(t)/\rho$, where $\chi_\rho(t)$ is the characteristic function of $I_\rho$. 
Then
\be
H_u=\frac2{\rho}\int_0^\rho \ft(t)\dd t=2\int_0^1 \ft(\rho s)\dd s\to 1,\rho\to\infty,
\ee
because $\ft(t)\to1/2$ as $t\to\infty$. Furthermore, using that $C_1$ is even, we find using \eqref{ga sq}
\begin{equation}\bs
\gamma^2 =& \rho^{-2}\int_{-\rho}^\rho\int_{-\rho}^\rho\text{sgn}(t_1)\text{sgn}(t_2)C_1(t_1-t_2)\dd t_1\dd t_2\\
=& 2\int_0^1\int_0^1 \big[C_1(\rho(s_1-s_2))-C_1(\rho(s_1+s_2))\big]\dd s_1\dd s_2.
%\\
%=& 2\int_0^1\int_{-s_2}^{1-s_2} C_1(\rho u)\dd u\dd s_2-2\int_0^1\int_{s_2}^{1+s_2}C_1(\rho u)\dd u\dd s_2.
\end{split}
\end{equation}
Since $\ik$ is compactly supported, \eqref{C1 def} implies $C_1(t)=O(1/t)$, $t\to\infty$. Hence
\be
\gamma^2 \le c \int_0^2 |C_1(\rho s)|\dd s= O(\ln\rho/\rho),\ \rho\to\infty.
\ee
Therefore $H_u/\ga\to\infty$ as $\rho\to\infty$. By \eqref{beta 1d}, this implies the following result.

\begin{lemma}\label{lem:1dpower}
The power, $1-\beta$, to detect a given edge at $\vx_0\in\s$ in 1D can be made as close to 1 as needed regardless of the SNR by selecting $\rho$ sufficiently large, assuming $0<\e\ll1$ is sufficiently small. 
\end{lemma}

Of course, in practice $\rho$ cannot increase without bounds. The limit on how large $\rho$ can be depends on how small $\e$ is. Hence, in practice, edges with SNR below some threshold cannot be detected.

Even if $\rho$ is fixed, the power of the test depends on the choice of $u$. In Section \ref{experiments}, we conduct simulated experiments where we consider $u(t) = \chi_\rho(t)\text{sgn}(t)$ and $u(t) = \chi_\rho(t) t$.

\section{Analysis in $\br^2$}\label{sec:anal2D}

\subsection{Edge detection in 2D}\label{2D_edge}

Define the ball $B_\rho(\chy_0):=\{\chy\in\br^2:\Vert\chy-\chy_0\Vert\le\rho\}$. We focus our attention on an $O(\epsilon)$-sized neighborhood around a point $\vx_0$ that satisfies assumption~\ref{ass:x0}, namely $\vx_0+\e B_\rho(\bold 0)$. Similarly to \eqref{Fu}, define three vectors, $F,G,H\in\br^2$,  
\begin{equation}\label{5.1}
\begin{split}
F& := \int_{B_\rho(\bold 0)} \chy f^{\text{rec}}( \chy)\dd \chy
= \int_{B_\rho(\bold 0)} \chy\ft(\chy)\dd \chy + \int_{B_\rho(\bold 0)} \chy N^{\text{rec}}( \chy )\dd \chy
=: H + G.
\end{split}
\end{equation}
As before, $H$ and $G$ are the deterministic and random parts of $F$, respectively. 

Comparing \eqref{5.1} with \eqref{Fu} we see that the 2D analog of $u(t)$ is $u(\cvy)=\cvy\chi_\rho(\cvy)$, where $\chi_\rho$ is the characteristic function of $B_\rho(\bold 0)$. The advantage of this choice is that its edge sensitivity is the same in all directions. Indeed, one has 
\be
H=\int_{B_\rho(\bold 0)} \chy\ft(\vec\Theta_0\cdot\chy)\dd \chy=\int_{s^2+t^2\le\rho^2} (t\vec\Theta_0+s\vec\Theta_0^\perp)\dd s\ft(t)\dd t.
\ee
Recall that $\vec\Theta_0$ is orthogonal to the edge at $\vx_0$. The integral containing the term $s\vec\Theta_0^\perp$ vanishes due to symmetry, and we find
\be\label{H no noise}
H=\int_{-\rho}^\rho \int_{|s|\le(\rho^2-t^2)^{1/2}} \dd s\, t\ft(t)\dd t\vec\Theta_0=\left[4\int_0^\rho t(\rho^2-t^2)^{1/2} \ft(t)\dd t\right]\vec\Theta_0.
\ee
Thus, in the absence of noise, $F=H$ is a vector normal to the edge, whose magnitude is proportional to the value of the jump. The coefficient of proportionality is the quantity in brackets in \eqref{H no noise}.

To develop our test statistic, we need the distribution of the random vectors $G$ and $F$. As in the preceding section, we first consider $G$. %Let us denote a typical point in $D\ni\chy=(\chyc_1,\chyc_2)$. 
It follows from Proposition \ref{prop2} in the appendix that  $G\sim \mathcal N(0,\hat C)$ is a Gaussian random vector with the covariance matrix
\be\label{vec covar}\bs
\hat C=(\hat C_{i,j}),\ \hat C_{i,j}=\Eb(G_iG_j)&=\Eb\left(\int_{B_\rho(\bold 0)} \chxc_i \Nrec(\chx)\dd\chx\int_{B_\rho(\bold 0)} \chyc_j\Nrec(\chy)\dd\chy\right)\\
&=\int_{B_\rho(\bold 0)}\int_{B_\rho(\bold 0)} \chxc_i\chyc_j\  C(\chx-\chy)\dd\chx \dd\chy,\ i,j\in\{1,2\},
\end{split}
\ee
where $C$ is given in \eqref{Cov main}. To compute the power of the edge detection test we first calculate the distribution of the random vector $F=H+G$. As before, it is easy to see that $F$ is also a bivariate Gaussian random vector, $F\sim \mathcal N(H,\hat C)$, and $\hat C$ is the covariance matrix \eqref{vec covar}. 
\begin{remark}\label{rem:swf} It follows that $\Eb(F)=H$. Thus, $F$ is an unbiased estimator of the vector $H$, which is normal to the edge and whose magnitude is proportional to the value of the jump. In this sense, since $\mathcal{S}=\text{sing\,supp}(f)$, then $\{(x_0,F):x_0\in \mathcal{S}\}$ can be regarded as a statistical analog of the wavefront set of $f$.
\end{remark}
\begin{remark}\label{rem:cov} As is easily seen from \eqref{Cov main}, if $\sigma(\al,\vec\al\cdot\vx_0)$ is a constant function, then $C(\chx)$ is a function of $|\chx|$, and the matrix $\hat C$ satisfies $\hat C=\nu^2\hat I_2$, where  $\nu_2=\hat C_{1,1}=\hat C_{2,2}>0$. Here $\hat I_2$ is the $2\times2$ identity matrix. We will use this fact later in the simulations in Section \ref{experiments}.
\end{remark}

Now we define a hypothesis test  for 2D edge detection to check whether $\vx_0\in\s$. 

\begin{theorem}[2D edge detection]\label{thm4.1}
Let $\vx_0$ satisfy assumption~\ref{ass:x0}. Consider an $O(\epsilon)$-sized neighborhood $\vx_0+\epsilon B_\rho(\bold 0)$ as above. Let $F$ be the random vector given in \eqref{5.1}. Consider the following two hypothesis:
\be\label{hypo 2D}
\hH_0 (\text{i.e. }x_0\notin \mathcal{S}): F \sim \Nc(0, \hat{C}); \quad \hH_1 (\text{i.e. }x_0\in \mathcal{S}): F \sim \Nc(H, \hat{C}),\ H \neq 0 .
\ee
Let the test statistic be $Z := \Vert F\Vert_{\hat{C}^{-1}}^2$. The LRT of size $\al$ for testing $\hH_0$ against $\hH_1$ rejects $\hH_0$ if $Z > -2\ln\al$. The power of the LRT is $1-\Upsilon_1(-2\ln\al)$, where $\Upsilon_1(x)$ is the CDF of a non-central $\chi^2$ random variable, $\chi_2^2(\mu)$, with the non-centrality parameter $\mu=\Vert H\Vert_{\hat{C}^{-1}}^2$. 
\end{theorem}

\begin{proof}
The proof of the theorem follows immediately from the LRT theory in Section~\ref{ssec:primer} and \eqref{5.1}. Here we just need to add two facts. First, the CDF of $\chi_2^2$ ($d=2$ degrees of freedom) is $\Upsilon_0(x)=1-e^{-x/2}$. The second fact is that under $\hH_1$, the test statistic $Z$ follows a non-central $\chi_2^2$ distribution, $Z\sim \chi_2^2(\mu)$, where $\mu=\Vert H\Vert_{\hat{C}^{-1}}^2$ is the noncentrality parameter.  
\end{proof}

To obtain an analog of Lemma~\ref{lem:1dpower}, suppose without loss of generality that the edge is normal to the $x_1$-axis (see \eqref{H no noise} and the paragraph that follows). Then $H=(h,0)$, and
\be
h/\rho^3=4\int_0^1 s(1-s^2)^{1/2}\ft(\rho s)\dd s\to 2/3,\ \rho\to\infty.
\ee
As before, $C(|\chx|)=O(|\chx|^{-1})$ as $|\chx|\to\infty$. By \eqref{vec covar},
\be\bs
\hat C_{i,j}=&\int_{|\chx|\le\rho}\int_{|\chy|\le\rho} \chxc_i\chyc_jC(\chx-\chy)\dd\chx \dd\chy\\
=&\rho^6\int_{|{\bold s}|\le1}\int_{|{\bold t}|\le1} s_i t_j  C(\rho({\bold s}-{\bold t}))\dd {\bold s} \dd {\bold t}=O(\rho^5),\ i,j\in\{1,2\}.
\end{split}
\ee
By simple linear algebra, the last two equations imply
\be\bs
\mu=\Vert H\Vert_{\hat{C}^{-1}}^2=H^T\hat C^{-1}H\ge \frac{\Vert H\Vert^2}{\la_{\max}(\hat C)}\ge \frac{\Vert H\Vert^2}{\text{trace}(\hat C)}\ge c\rho,
\end{split}
\ee
which proves the desired assertion. Here we have used that $\hat C$ is positive semidefinite. This proves the following lemma.

\begin{lemma}\label{lem:2dpower}
The power, $1-\beta$, to detect a given edge at $\vx_0\in\s$ in 2D can be made as close to 1 as needed regardless of the SNR by selecting $\rho$ sufficiently large, assuming $0<\e\ll1$ is sufficiently small. 
\end{lemma}

The same comment as after Lemma~\ref{lem:1dpower}, which concerned the 1D case, applies in the 2D case as well.

\subsection{Uncertainty quantification}\label{ssec:uq}
Recall that for any Gaussian random vector $X\in\br^2$, $X\sim \Nc(\mu,\Sigma)$, the contour of constant probability density is an ellipse:
\be\label{Er def}\bs
    E_r(\mu) :=&\big\{X:  (2\pi)^{-1} |\hat{C}|^{-1/2} \exp\big(-\Vert X-\mu\Vert_{\hat{C}^{-1}}^2/2\big)=c_0\big\}\\
    =&\{X: \Vert X-\mu\Vert_{\hat{C}^{-1}}^2=r\},\ r=-2\ln\big(2\pi |\hat{C}|^{1/2}c_0\big).
\end{split}
\ee
%Indeed, if $\lambda_i$ and $u_i$, $i=1,2,\dots,d$, are the eigenvalues and eigenvectors of $\hat{C}$, respectively, then the principal axes of the ellipse are along the $u_i$, and their respective lengths are given by $\sqrt{\lambda_i}$.  

Recall that a {\it confidence region}, $\text{CR}(F)$, is a mapping from the parameter space, $\br^2$, into the set of all subspaces of $\br^2$. More precisely, it maps $\br^2\ni F\to \text{CR}(F)\subset \br^2$. It is important to emphasize that $\text{CR}(F)$ is a random set due to the randomness in $F$. We say that $\text{CR}(F)$ is at the confidence level $1-\al$ if $\Pb(H\in \text{CR}(F))\ge 1-\al$ and denote this by $\text{CR}_\al(F)$. To clarify, the condition $\Pb(H\in \text{CR}_\al(F))\ge 1-\al$ means that if $F$ is randomly drawn from $\Nc(H,\Sigma)$ a large number of times, then $H\in \text{CR}_\al(F)$ in at least $(1-\al)\%$ of the cases. See \cite[Section 14.4.2]{LeRo22} and \cite[Sections 9.1, 9.2]{CB_book} for more details.

Based on the above, we search for a confidence region of the form $E_r(F)$, where $r$ is determined from $\al$. By \eqref{hypo 2D}, the PDF of $F$ is $f(u;H,\hat C) = (2\pi)^{-1} |\hat{C}|^{-1/2} \exp\big(-\Vert u-H\Vert_{\hat{C}^{-1}}^2/2\big)$. Hence, by definition, we should solve the following equation
\be\label{CR v1}
\inf_{H\in\br^2}\int_{\br^2} \chi(H; E_r(u))f(u;H,\hat C)\dd u=1-\al.
\ee
Here $\chi(H; E_r(u))=1$ if $H\in E_r(u)$ and $\chi(H; E_r(u))=0$ if $H\not\in E_r(u)$. By \eqref{Er def}, $\chi(H; E_r(u))$ is a function of only $u-H$. Likewise, $f(u;H,\hat C)$ is a function of only $u-H$. Hence the integral in \eqref{CR v1} is independent of $H$ and we can just set $H=0$. This gives a simplified equation for $r$ in terms of $\al$:
\be\label{eq:CRF}
\frac1{2\pi |\hat{C}|^{1/2}}\int_{\Vert u\Vert_{\hat{C}^{-1}}^2\le r}  \exp\big(-\Vert u\Vert_{\hat{C}^{-1}}^2/2\big)\dd u=1-\al,
\ee
i.e. $\Pb(\Vert u\Vert_{\hat{C}^{-1}}^2\le r)=1-\alpha$. Changing variables $v=\hat C^{-1/2}u$, we see that $v\sim \Nc(0,\hat I_2)$ and hence $\Vert u\Vert_{\hat{C}^{-1}}^2=v^Tv=\lVert v\rVert^2\sim \chi_2^2$. Thus $r$ is found from the transformed equation $\Pb(\lVert v\rVert^2\leq r)=1-\alpha$, which gives the familiar $r=\Upsilon_0^{-1}(1-\alpha)=-2\ln\al$. This shows that the desired confidence region is the ellipsoid  $E_r(F)$, where $r=-2\ln\al$.

For the convenience of the reader, we have just derived the confidence region from first principles. An easier and more direct approach, which gives the exact same $\text{CR}_\al(F)$, is based on statistical test inversion \cite[Theorem 9.2.2]{CB_book}. It goes as follows. Suppose $F$ is our observation. For each $H\in\br^d$, we test the null hypothesis $\hH_0:F\sim\Nc(H,\Sigma)$ versus the alternative: $\hH_1:F\not\sim\Nc(H,\Sigma)$ at the level $1-\al$. Then the confidence region $\text{CR}_\al(F)$ is the set of all $H$ for which the null hypothesis is accepted. By writing out the condition on $H$ and $F$ which results in the acceptance of $\hH_0$ we recover the same set $\text{CR}_\al(F)$. 

To summarize, we obtained the following result.

\begin{theorem}
A confidence region for the true parameter value $H$ at the level $1-\al$, denoted by $\text{CR}_{\alpha}(F)$, is given by the following random set:
\be
\text{CR}_{\alpha}(F)=\big\{H\in \Rb^2:
\Vert F - H\Vert_{\hat{C}^{-1}} \leq (-2\ln\al)^{1/2} \big\}.
\ee
\end{theorem}

\section{Experiments}\label{experiments}\label{sec:exps}

We now conduct experiments to validate our theory. For reconstruction, we use the formula in \eqref{recon-orig}, and to generate Radon transform (line integral) data we use the Matlab function ``radon." We consider first the 1D case described in Section \ref{first_prin}.

\subsection{1D edge detection}\label{ssec:1dexp}
In this example, we consider the reconstruction target shown in figure \ref{fig1}, which is the characteristic function of a ball with radius $R=0.345$ and center $(0,-0.1)$. Throughout this section, we simulate noise $\eta \sim U(0,\sigma)$ from a uniform distribution with mean zero and standard deviation $\sigma$. We set $\epsilon = 0.007$ and the jump size is $\Delta f(\vx) = 1$ for any $\vx\in\s$. For the experiments, we select $\vx_0=(R,-0.1)$. We set $\kappa = 2\pi$, with $\kappa |\vx_0|$ irrational in line with Assumption~\ref{ass:x0}\eqref{x0_1}. The ball is slightly off center so that Assumption~\ref{ass:x0}\eqref{x0_5} is satisfied as well. The integral in \eqref{Fu} in this case is taken along the red line segment shown in figure \ref{fig1} to test the presence of an edge. 
\begin{figure}[!h]
\centering
\begin{subfigure}{0.3\textwidth}
\includegraphics[width=0.9\linewidth, height=3.2cm, keepaspectratio]{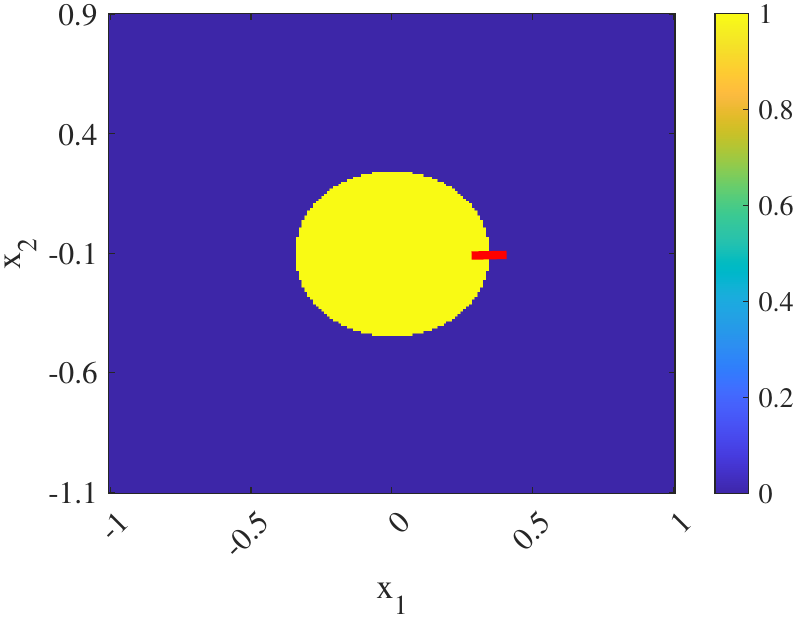}
\subcaption{phantom} \label{2a}
\end{subfigure}
\begin{subfigure}{0.3\textwidth}
\includegraphics[width=0.9\linewidth, height=3.2cm, keepaspectratio]{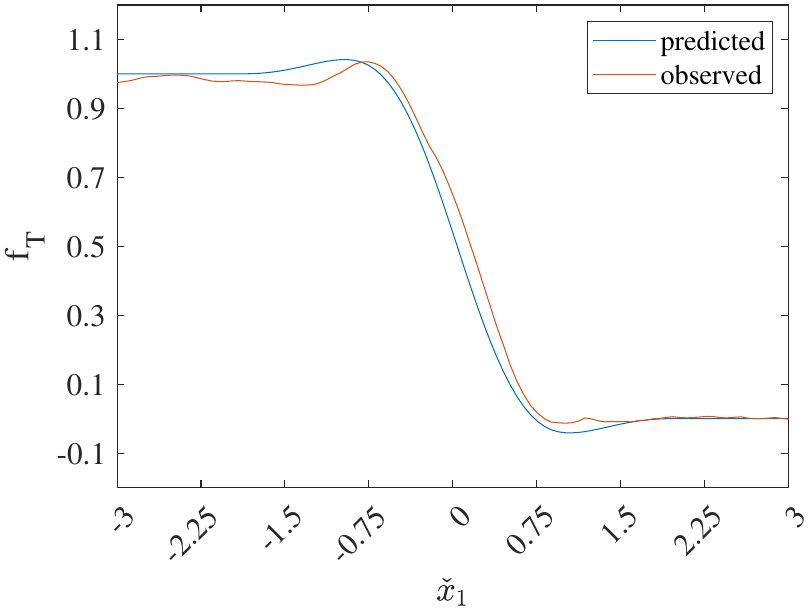}
\subcaption{line profile} \label{2b}
\end{subfigure}
\caption{(a) Test phantom is a ball. The red horizontal line, which represents the edge detection window $I_\rho$, has length $6\epsilon$ (i.e., $\rho = 3$). (b) A local 1D profile of the reconstruction of the phantom along the red horizontal segment and our prediction, $\ft$. The plot in (b) is shown using the checked coordinates which are scaled to $\e$ and centered on zero.}
\label{fig1}
\end{figure}

For our first example, we set $\sigma^2 = 3$ and $u(t) = \chi_\rho(t) t$. In this case, the distribution of the random variable $F_{u}$ in the absence and presence of an edge are shown in figure \ref{fig3}. To generate the histograms, we sampled $G_{u}$ (the random part of the reconstruction) $2\cdot 10^3$ times. In figure \ref{3a}, $10^3$ samples were used to calculate the null distribution histogram (no edge), and the remaining $10^3$ were used to calculate the shifted ``edge" histogram. We calculated $H_u$ using the $\ft$ plot of figure \ref{2b}.
\begin{figure}
\centering
\begin{subfigure}{0.3\textwidth}
\includegraphics[width=0.9\linewidth, height=3.2cm, keepaspectratio]{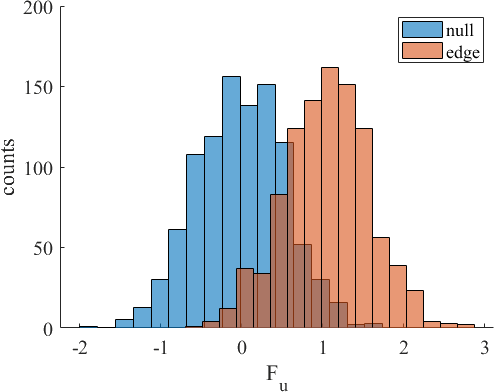}
\subcaption{$F_u$ histograms} \label{3a}
\end{subfigure}
\begin{subfigure}{0.3\textwidth}
\includegraphics[width=0.9\linewidth, height=3.2cm, keepaspectratio]{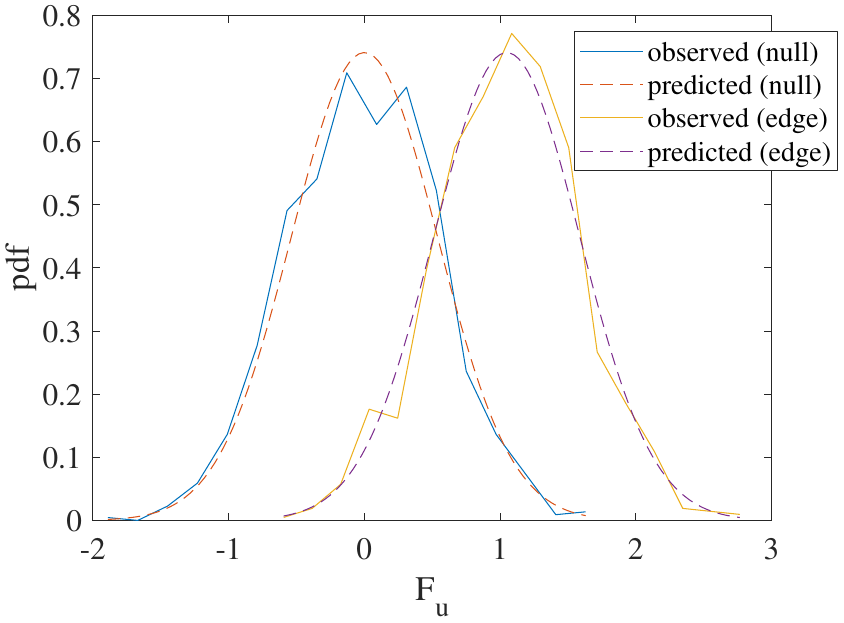}
\subcaption{prediction vs observed} \label{3b}
\end{subfigure}
\begin{subfigure}{0.3\textwidth}
\includegraphics[width=0.9\linewidth, height=3.2cm, keepaspectratio]{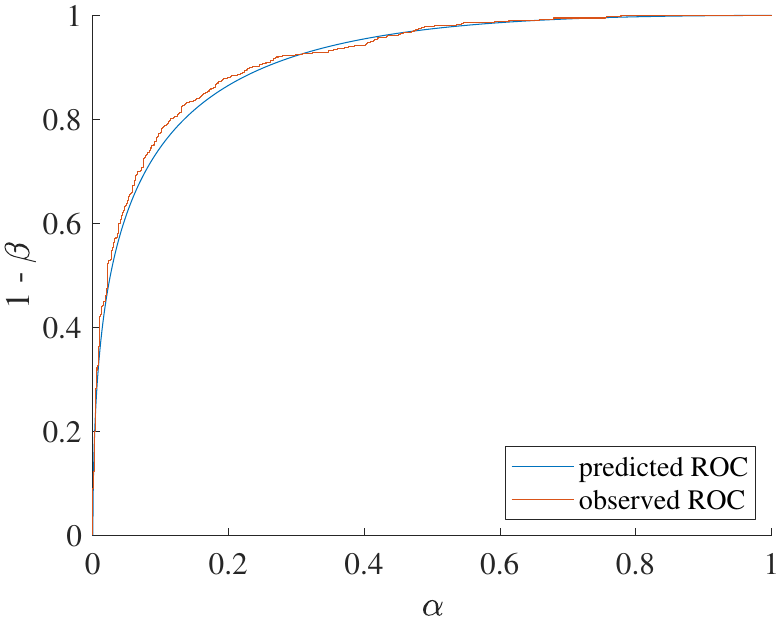}
\subcaption{ROC plots} \label{3c}
\end{subfigure}
\caption{(A) and (B) - example $F_u$ histograms and PDFs when $u(t) = \chi_\rho(t) t$. Recall that in the 1D case, the random variable $F_u$ is proportional to the jump magnitude across a possible edge location.  (C) - predicted and observed ROC curves.}
\label{fig3}
\end{figure}
In figure \ref{3b}, the predicted PDF curves based on the theory in Section~\ref{first_prin} are plotted against the PDFs calculated using the histograms in figure \ref{3a}, and they match well. 

%To calculate the relation between $\alpha$ and $\beta$, we treat the edge detection as a binary classification problem, where the null samples are \tred{negatives} and the edge samples are \tred{positives}. \tred{The terminology is not clear.} Based on this model, we plot the Receiver Operator Characteristic (ROC) curves in figure \ref{3c} using \eqref{stat_test}. The result matches well the empirically obtained curve. Recall that the ROC curve is the graph of the function $\bt=\bt(\al)$. For example, $\alpha = 0.05$ gives $\beta = 0.64$, i.e., the probability the edge is detected correctly is $64\%$. 
We also calculate the Area Under the ROC Curve (AUC) to test the effectiveness of the proposed method. Letting $1-\bt(\alpha)$ be the function which defines the ROC curve, the AUC is computed as $\text{AUC} = 1-\int_0^1 \bt(\alpha)\dd\al$. The value predicted using \eqref{beta 1d} is $\text{AUC} = 0.92$, which is consistent with the histograms in figure \ref{3b} having mild overlap. The observed AUC also equals 0.92 to two significant figures. 
\begin{figure}
\centering
%\begin{subfigure}{0.3\textwidth}
%\includegraphics[width=0.9\linewidth, height=3.2cm, keepaspectratio]{lp_1_1}
%\subcaption{$\text{AUC} = 1$, $\sigma = 0.35$} \label{5a}
%\end{subfigure}
\begin{subfigure}{0.24\textwidth}
\includegraphics[width=0.9\linewidth, height=3.2cm, keepaspectratio]{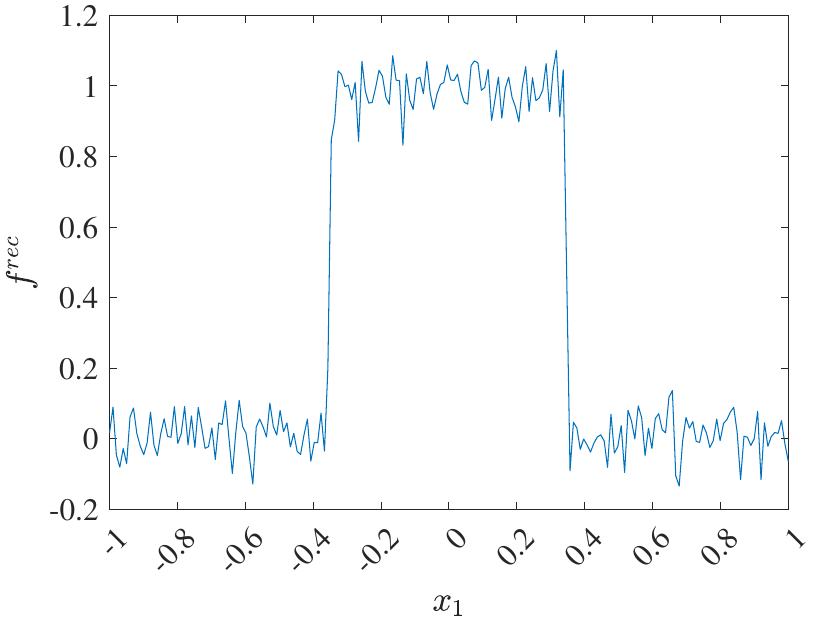}
\subcaption{$\text{AUC} = 0.99$, $\sigma = 0.87$} \label{5b}
\end{subfigure}
\begin{subfigure}{0.24\textwidth}
\includegraphics[width=0.9\linewidth, height=3.2cm, keepaspectratio]{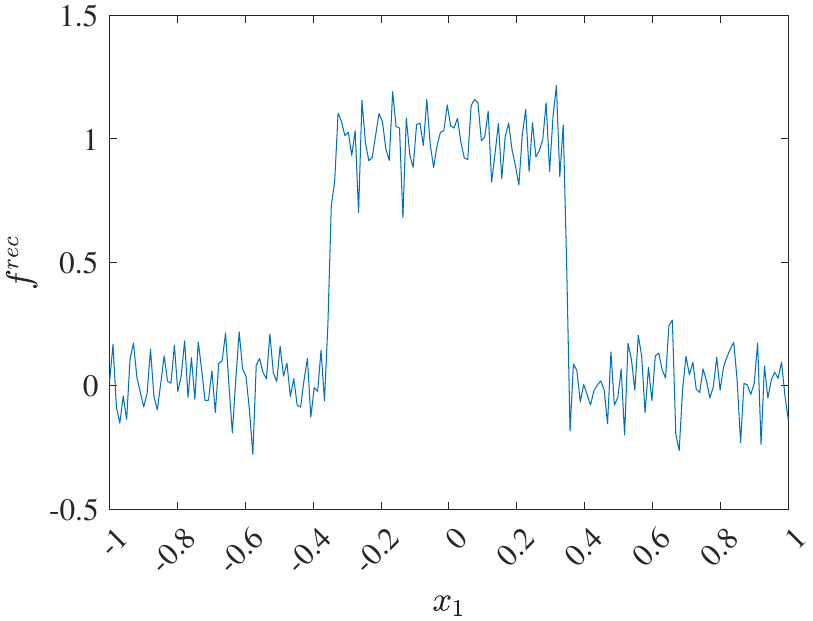}
\subcaption{$\text{AUC} = 0.93$, $\sigma = 1.73$} \label{5c}
\end{subfigure}
\begin{subfigure}{0.24\textwidth}
\includegraphics[width=0.9\linewidth, height=3.2cm, keepaspectratio]{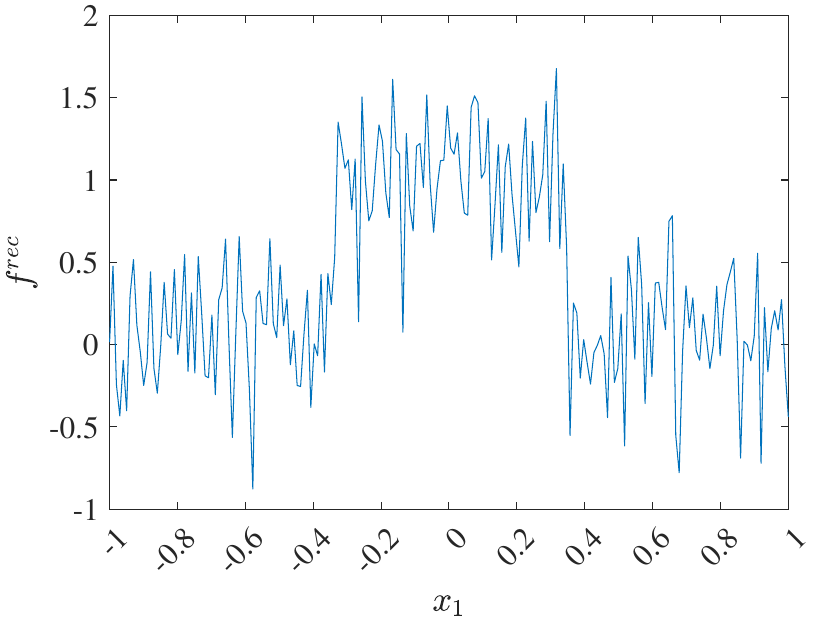}
\subcaption{$\text{AUC} = 0.69$, $\sigma = 5.2$} \label{5d}
\end{subfigure}
%\begin{subfigure}{0.3\textwidth}
%\includegraphics[width=0.9\linewidth, height=3.2cm, keepaspectratio]{lp_7_1}
%\subcaption{$\text{AUC} = 0.56$, $\sigma = 17.3$} \label{5e}
%\end{subfigure}
\begin{subfigure}{0.24\textwidth}
\includegraphics[width=0.9\linewidth, height=3.2cm, keepaspectratio]{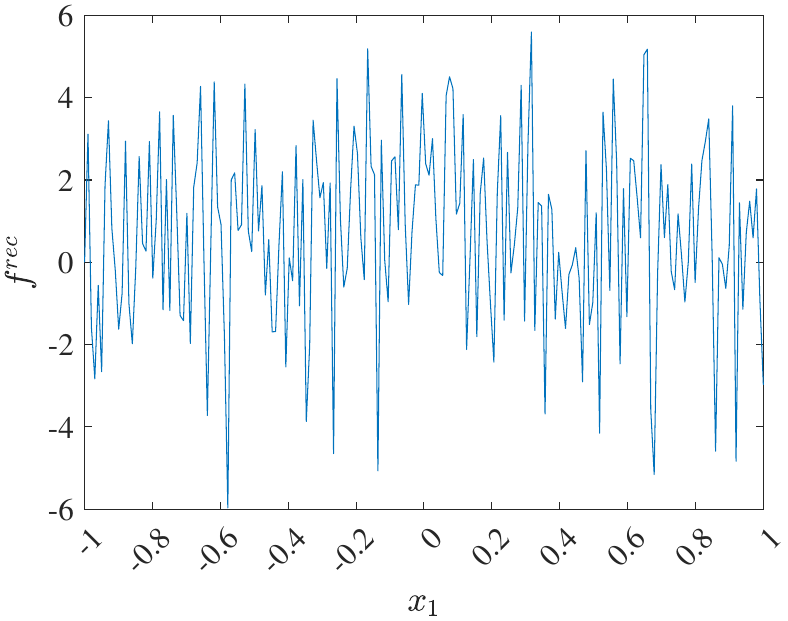}
\subcaption{$\text{AUC} = 0.53$, $\sigma = 34.6$} \label{5f}
\end{subfigure}
\caption{1D line profiles (along $\{x_2 = -0.1\}$) of reconstructions of the image phantom for varying $\sigma$.}
\label{fig5}
\end{figure}
\begin{figure}
\centering
\begin{subfigure}{0.3\textwidth}
\includegraphics[width=0.9\linewidth, height=3.2cm, keepaspectratio]{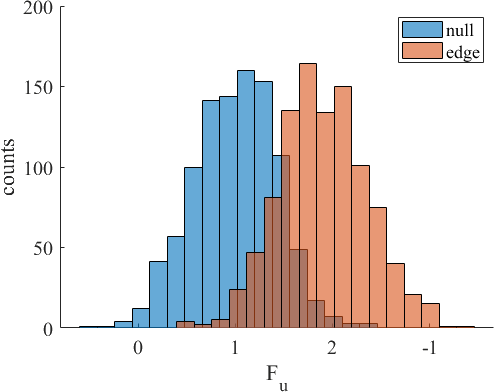}
\subcaption{$F_u$ histograms} \label{7a}
\end{subfigure}
\begin{subfigure}{0.3\textwidth}
\includegraphics[width=0.9\linewidth, height=3.2cm, keepaspectratio]{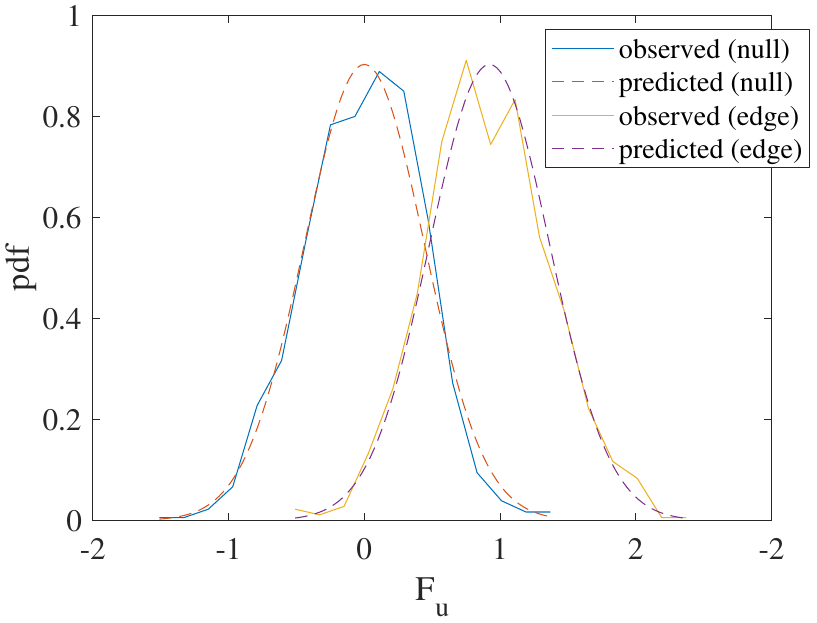}
\subcaption{prediction vs observed} \label{7b}
\end{subfigure}
\begin{subfigure}{0.3\textwidth}
\includegraphics[width=0.9\linewidth, height=3.2cm, keepaspectratio]{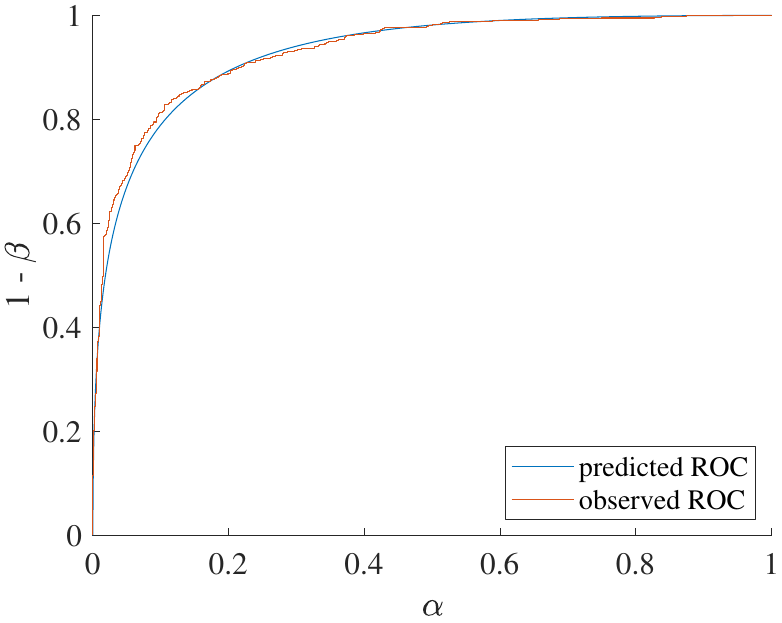}
\subcaption{ROC plots} \label{7c}
\end{subfigure}
\caption{(A) and (B) - example $F_u$ histograms and PDFs when $u(t) = \chi_\rho(t)\text{sgn}(t)$. These plots are similar to the ones observed in Fig. 3 illustrating that $u(t)=\chi_{\rho}(t)t$ and $u(t)=\chi_{\rho}\text{sgn}(t)$ are both equally good choices to construct $F_{u}$.(C) - predicted and observed ROC curves.}
\label{fig7}
\end{figure}

To show how the predicted AUC score relates to edge reconstruction quality for different noise levels, in figure \ref{fig5} we plot reconstructions of the phantom along the $x_1$ axis for varying noise levels, $\sigma$. For example, when $\text{AUC} \approx 0.5$, the noise level is too great to see the edge and the probability of the edge being detected correctly reduces to $1/2$, i.e. that of a random classifier (the dashed line in Figure~\ref{fig:roc}). Conversely, when $\sigma\to0$ and, consequently, $\text{AUC} \to 1$, the edge is clearly visible.

We now consider the case when $u(t) = \chi_\rho(t)\text{sgn}(t)$ and $\sigma^2 = 3$ as before. See figure \ref{fig7}. The agreement between the predicted and observed PDFs is good, as in the last example, and the ROC curves match well. The predicted and observed AUC scores are both equal $\text{AUC} = 0.93$, which is only marginally higher than in the last example when $u(t) = \chi_\rho(t)t$. When $\alpha = 0.05$, the corresponding value of $1 - \beta$ is $1 - \beta = 0.67$, which is $3\%$ higher than in the last example. Thus, with all other variables fixed (e.g., $\sigma$), the kernels $u(t) = \chi_\rho(t)\text{sgn}(t)$ and $u(t) = \chi_\rho(t)t$ offer similar levels of statistical power.
\begin{figure}[!h]
\centering
%\begin{subfigure}{0.24\textwidth}
%\includegraphics[width=0.9\linewidth, height=3.2cm, keepaspectratio]{plot_1_1}
%\subcaption{$\text{AUC} = 1$, $\sigma = 0.35$}\label{ba1a}
%\end{subfigure}
\begin{subfigure}{0.24\textwidth}
\includegraphics[width=0.9\linewidth, height=3.2cm, keepaspectratio]{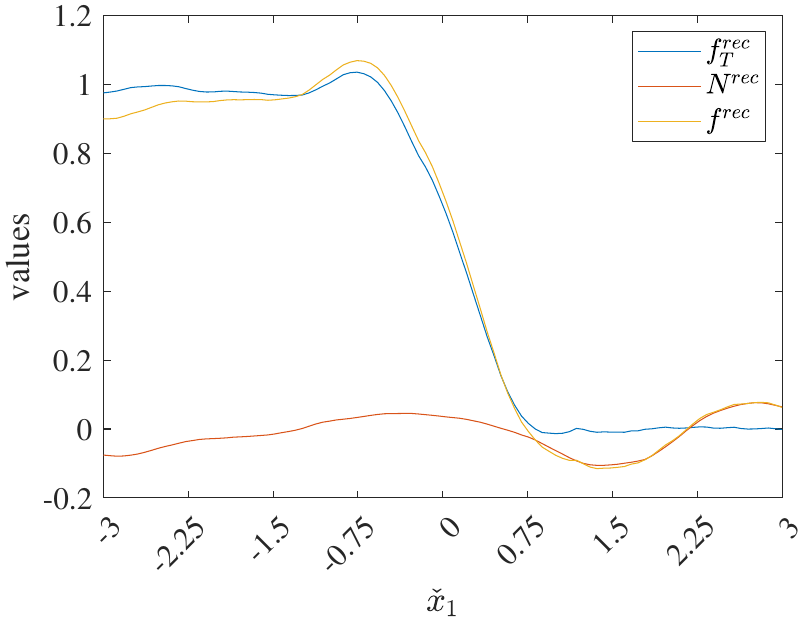}
\subcaption{$\text{AUC} = 0.99$, $\sigma = 0.87$} \label{be1a}
\end{subfigure}
\begin{subfigure}{0.24\textwidth}
\includegraphics[width=0.9\linewidth, height=3.2cm, keepaspectratio]{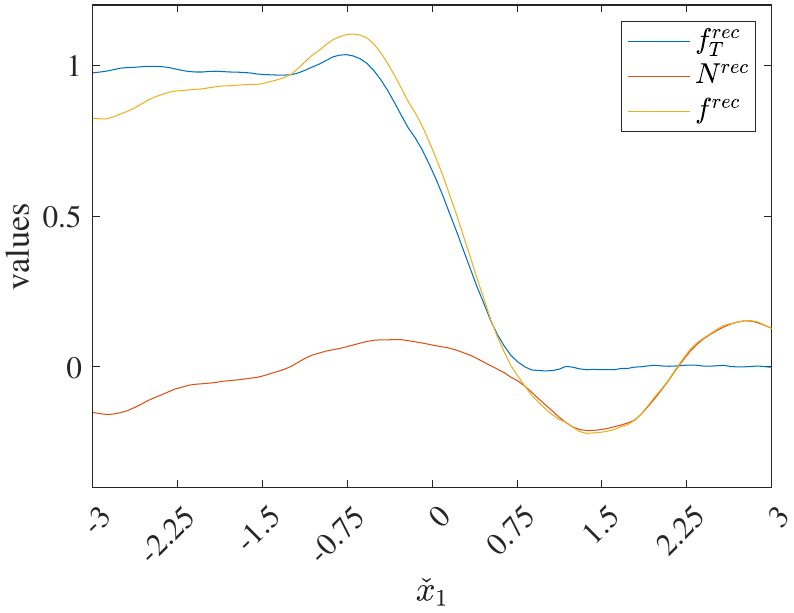}
\subcaption{$\text{AUC} = 0.93$, $\sigma = 1.73$} \label{bar1}
\end{subfigure}
%\begin{subfigure}{0.24\textwidth}
%\includegraphics[width=0.9\linewidth, height=3.2cm, keepaspectratio]{plot_4_1}
%\subcaption{$\text{AUC} = 0.77$, $\sigma = 3.5$} \label{ber1a}
%\end{subfigure}
\begin{subfigure}{0.24\textwidth}
\includegraphics[width=0.9\linewidth, height=3.2cm, keepaspectratio]{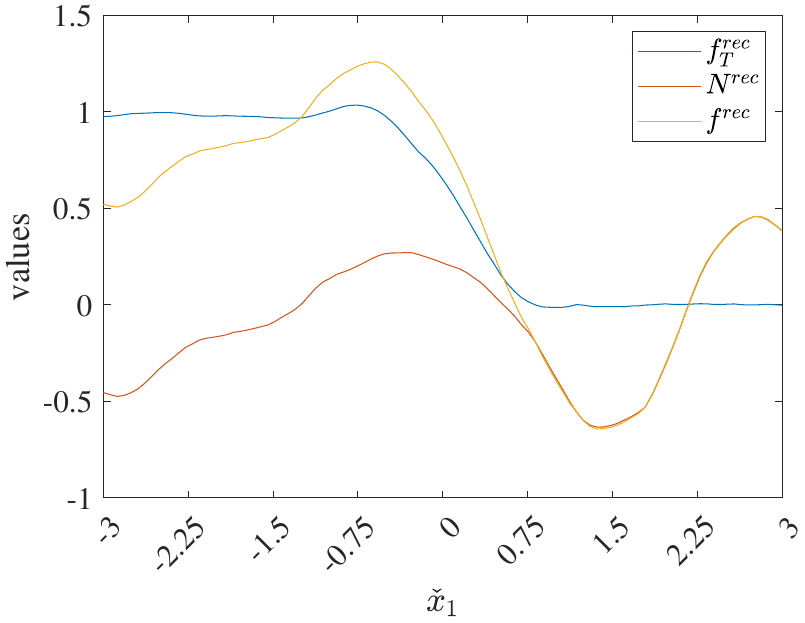}
\subcaption{$\text{AUC} = 0.69$, $\sigma = 5.2$}\label{ba1b}
\end{subfigure}
%\begin{subfigure}{0.24\textwidth}
%\includegraphics[width=0.9\linewidth, height=3.2cm, keepaspectratio]{plot_6_1}
%\subcaption{$\text{AUC} = 0.58$, $\sigma = 12.1$} \label{be1b}
%\end{subfigure}
%\begin{subfigure}{0.24\textwidth}
%\includegraphics[width=0.9\linewidth, height=3.2cm, keepaspectratio]{plot_7_1}
%\subcaption{$\text{AUC} = 0.56$, $\sigma = 17.3$} \label{bar1}
%\end{subfigure}
\begin{subfigure}{0.24\textwidth}
\includegraphics[width=0.9\linewidth, height=3.2cm, keepaspectratio]{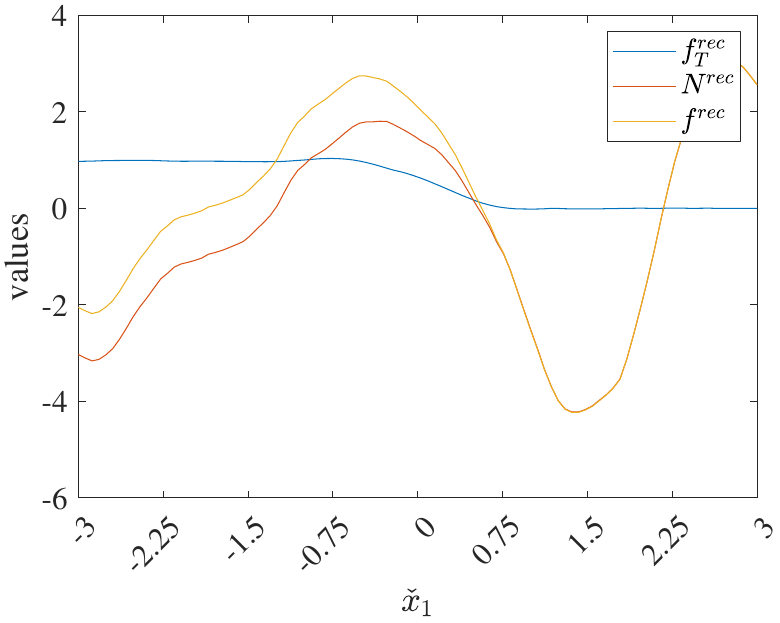}
\subcaption{$\text{AUC} = 0.53$, $\sigma = 34.6$} \label{ber1b}
\end{subfigure}
\caption{Local 1D line profiles of the reconstructions when $u(t) = \chi_\rho(t)\text{sgn}(t)$ at varying AUC levels. The blue curves show the noiseless reconstruction, the red curves are reconstructions just from noise, and the orange curves are the sum of both. Note that with increasing $\sigma$, the reconstruction from pure noise (red plot) dominates the noise-less reconstruction (blue plot) which is reflected
in the worsening AUC score as noise increases. }
\label{fig12_sgn}
\end{figure}

In figure \ref{fig5} we plot line profiles of the reconstruction on the full image scale. In figure \ref{fig12_sgn}, we plot the reconstruction values locally, i.e., along the red line profile shown in figure \ref{2a}, for varying noise levels and give the corresponding AUC values in the figure subcaptions.
We observe a similar effect as on the macro scale, except in this case, the reconstruction from purely noise is continuous as we are working in local neighborhoods, and $N^{\text{rec}}(\cvx)$ is a continuous GRF in rescaled coordinates.

\subsection{2D experiments}\label{ssec:2dexp}
Here, we conduct experiments to validate our 2D edge detection theory using the same setting as in the 1D experiments. We consider reconstruction of the disk in figure \ref{fig1}. To calculate $F=(F_1,F_2)$, we integrate locally over the domain $\vx_0+\e B_\rho(\bold 0)$, where $\vx_0=(R,0)$, $\rho = 3$ and $\e = 0.007$ as in the previous examples in 1D. We simulate noise $\eta \sim U(0,\sigma)$ from a uniform distribution with mean zero and standard deviation $\sigma$, and we set the jump size $\Delta f(\vx_0) = 1$. The deterministic part of the reconstruction on $D$ is shown in figure \ref{fig11} using the checked coordinates scaled to $\epsilon$ and centered on zero for better visualization. We see good agreement between our prediction and the observed reconstruction. As we are working locally, the boundary appears flat. In this section and in line with the theory of Section \ref{2D_edge}, we consider only linear weights, $u$.
\begin{figure}[!h]
\centering
\begin{subfigure}{0.3\textwidth}
\includegraphics[width=0.9\linewidth, height=3.2cm, keepaspectratio]{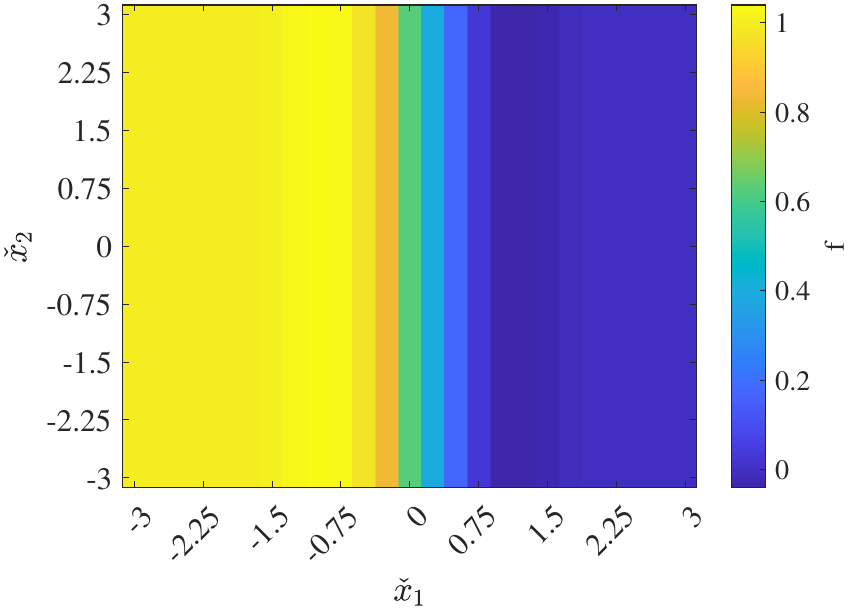}
\subcaption{prediction} \label{11a}
\end{subfigure}
\begin{subfigure}{0.3\textwidth}
\includegraphics[width=0.9\linewidth, height=3.2cm, keepaspectratio]{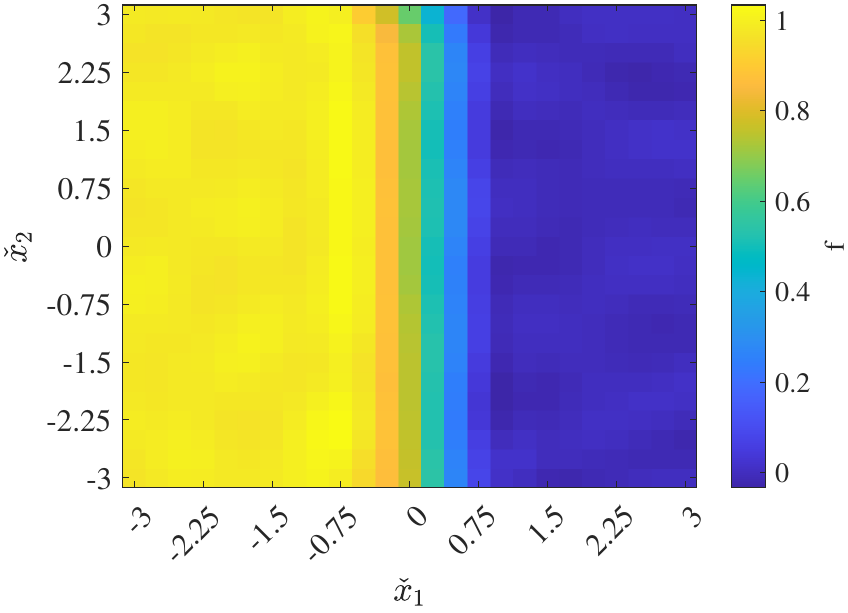}
\subcaption{observed} \label{11b}
\end{subfigure}
\caption{Deterministic part of the reconstruction. (a) - predicted local reconstruction of spherical phantom on $D:=[-\rho, \rho]^2$. (b) observed local reconstruction on $D$. They are the 2D analogs of the 1D line profiles through the reconstructions as presented in Figure~\ref{fig1}(b).}
\label{fig11}
\end{figure}
Using the image in figure \ref{11a}, we estimate the deterministic part of $F$, namely $H = (H_1,H_2) = (1.23, 0)$, which matches with the observed $H_o = (1.19,0)$ using the reconstruction in figure \ref{11b} to calculate the integrals.

We now focus on $G = (G_1,G_2)$, which is calculated using the random part of the reconstruction. For comparison, we set $\sigma^2 = 3$ as in the 1D examples. In this example, using \eqref{vec covar}, we predict that $G$ is normally distributed with mean zero and covariance $\hat C = \begin{pmatrix} 0.074 & 0\\ 0 & 0.074\end{pmatrix}$. Since $\sigma$ is a constant function, $\hat C=\nu^2\hat I_2$, $\nu^2=0.14$ (see Remark~\ref{rem:cov}). To test our prediction, we generate $10^4$ samples of $G$ by reconstructing from purely random noise draws on $D$ (using the noise distribution described above) and calculate the weighted integrals as in \eqref{5.1}. The observed $G$ is this case has mean $H_o = (-0.0018,-0.0026) \approx 0$ to two significant figures, and covariance $\hat C_o = \begin{pmatrix} 0.074 & 0\\ 0 & 0.075\end{pmatrix}$, again working to two significant figures. The histogram of the observed $G$ appears normal and matches well with the predicted PDF. See figure \ref{fig12}. The same is true for the shifted PDFs when there is an edge present. To calculate the shifted histograms, we generated $10^4$ further $G$ samples in the same way. In figure \ref{fig13}, we show the same plots as in figure \ref{fig12} but on the same grid (as Gaussian mixtures) to show better the separation between the shifted and null PDFs.

\begin{figure}[!h]
\centering
\begin{subfigure}{0.3\textwidth}
\includegraphics[width=0.9\linewidth, height=3.2cm, keepaspectratio]{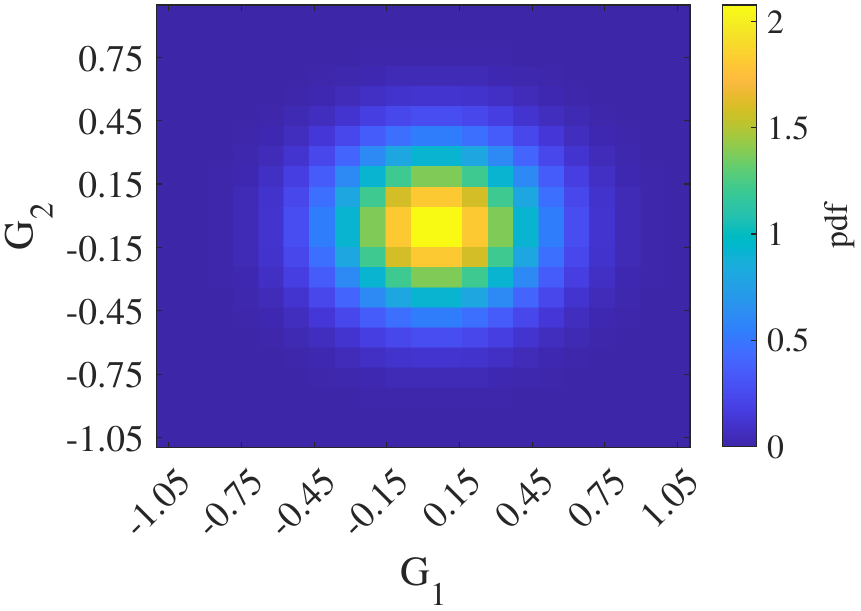}
\label{12a}
\end{subfigure}
\begin{subfigure}{0.3\textwidth}
\includegraphics[width=0.9\linewidth, height=3.2cm, keepaspectratio]{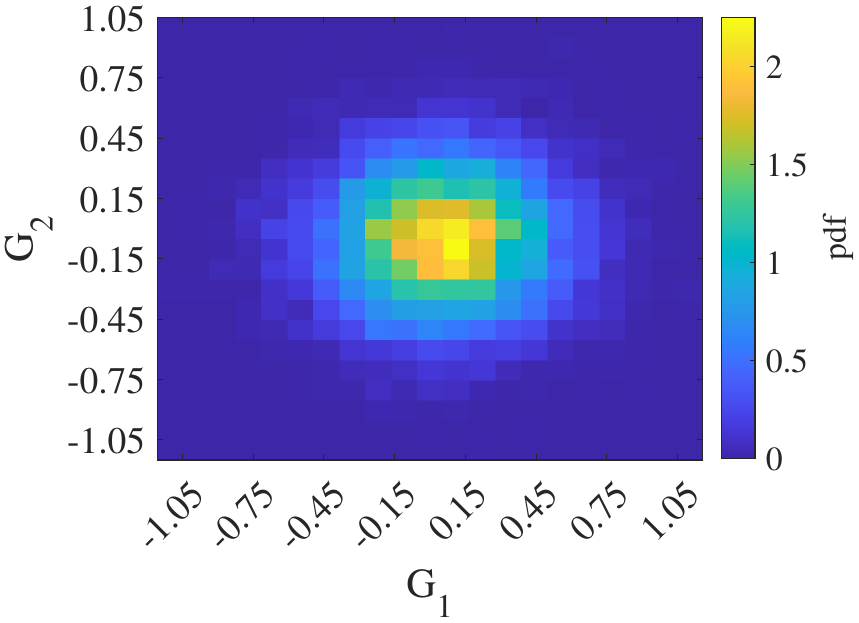}
\label{12b} 
\end{subfigure}
\begin{subfigure}{0.3\textwidth}
\includegraphics[width=0.9\linewidth, height=3.2cm, keepaspectratio]{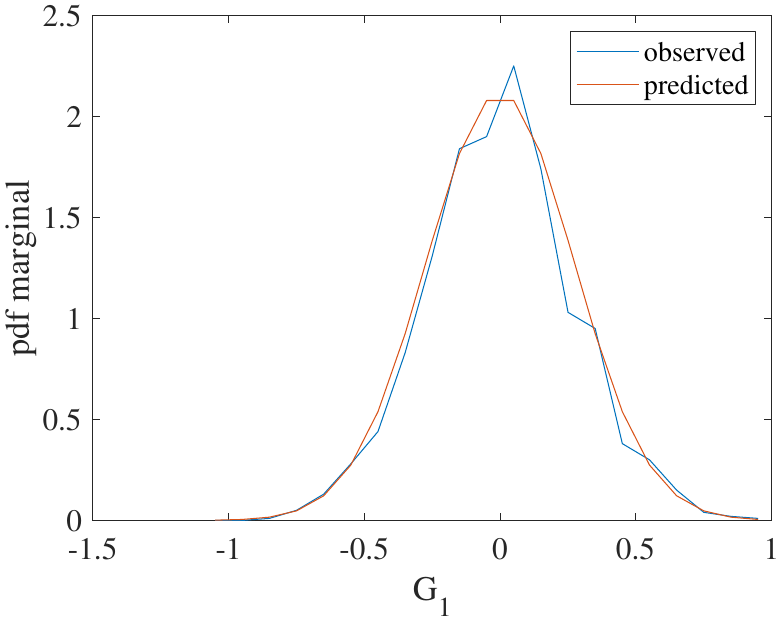}
\label{12b_1} 
\end{subfigure}
\\
\begin{subfigure}{0.3\textwidth}
\includegraphics[width=0.9\linewidth, height=3.2cm, keepaspectratio]{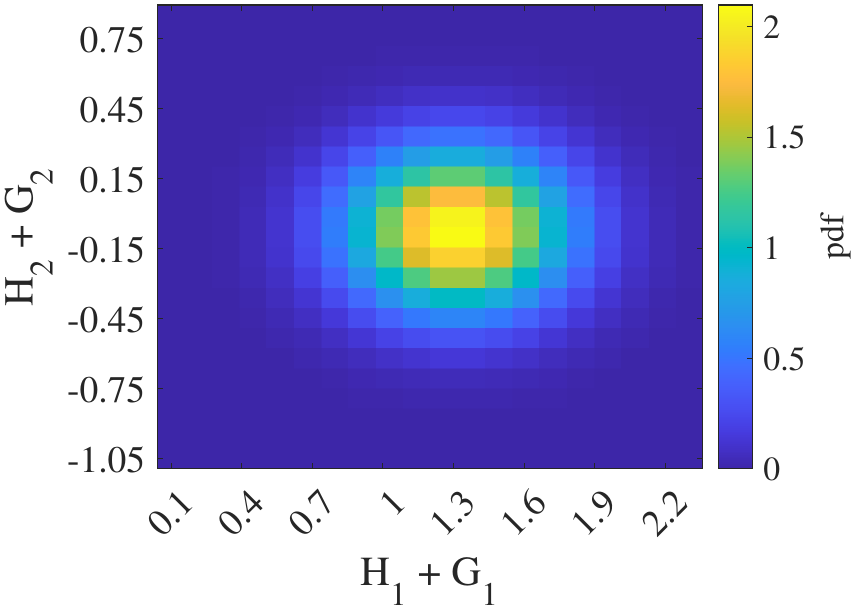}
\subcaption*{prediction} \label{12c}
\end{subfigure}
\begin{subfigure}{0.3\textwidth}
\includegraphics[width=0.9\linewidth, height=3.2cm, keepaspectratio]{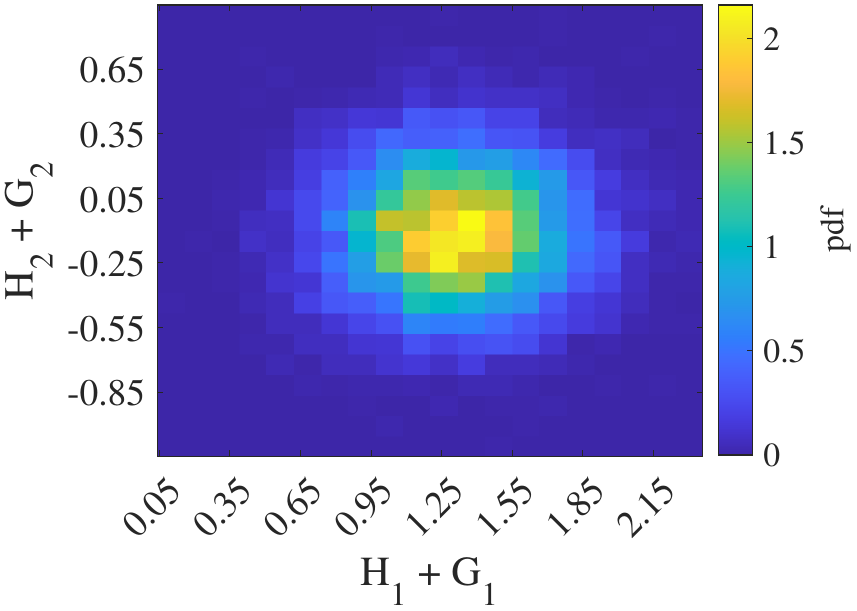}
\subcaption*{observed} \label{12d}
\end{subfigure}
\begin{subfigure}{0.3\textwidth}
\includegraphics[width=0.9\linewidth, height=3.2cm, keepaspectratio]{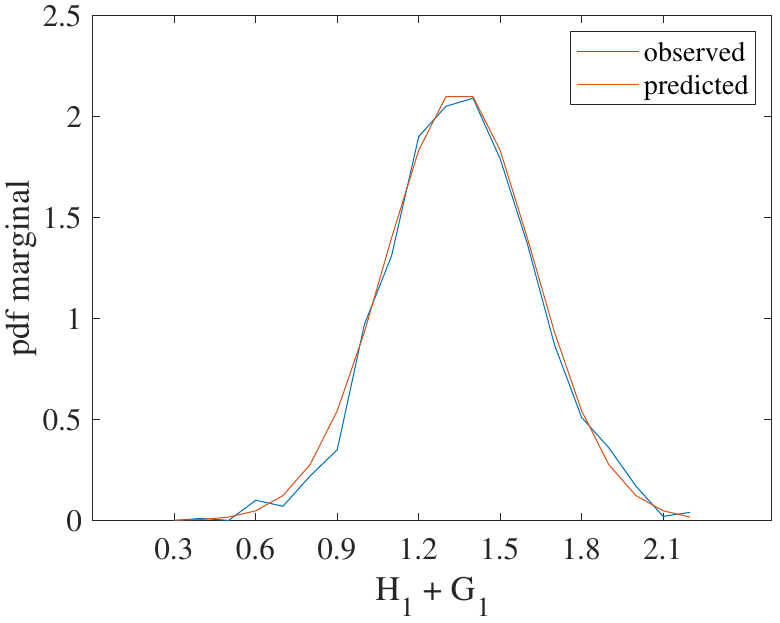}
\subcaption*{1D marginals} \label{12d_1} 
\end{subfigure}
\caption{Predicted and observed PDFs of $F$. Top row - null distributions ($F=G$), in this case there is no edge and $H=0$. Bottom row - shifted (edge) distributions ($F = G + H$). The right column shows 1D marginals of the PDFs (through the $F_1$ axis).}
\label{fig12}
\end{figure}

\begin{figure}
\centering
\begin{subfigure}{0.3\textwidth}
\includegraphics[width=0.9\linewidth, height=3.2cm, keepaspectratio]{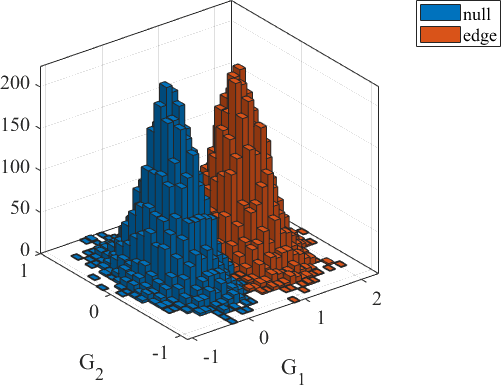}
\subcaption{histogram (observed)} \label{13a}
\end{subfigure}
\begin{subfigure}{0.3\textwidth}
\includegraphics[width=0.9\linewidth, height=3.2cm, keepaspectratio]{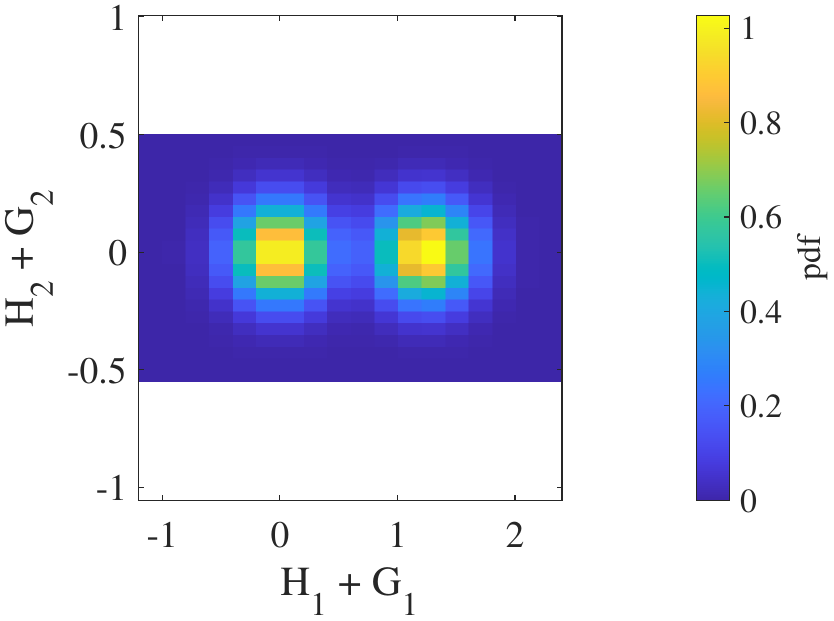}
\subcaption{prediction} \label{13b}
\end{subfigure}
\begin{subfigure}{0.3\textwidth}
\includegraphics[width=0.9\linewidth, height=3.2cm, keepaspectratio]{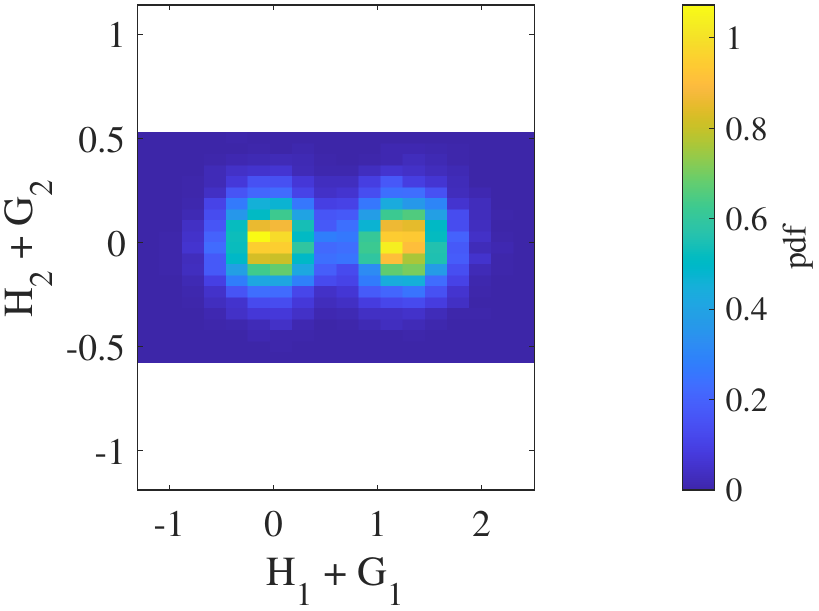}
\subcaption{observed} \label{13c}
\end{subfigure}
\caption{In (b) and (c) we show the PDF plots from figure \ref{fig12} side-by-side with a histogram of the observed $F$ values as a bar chart in (a). Observe that the null and shifted PDFs and histograms are well separated; thus illustrating that the edge will be detected with a high probability.}
\label{fig13}
\end{figure}

Let $\alpha$ denote the type I error of the test. By Theorem~\ref{thm4.1}, $100(1-\alpha)\%$ of the draws from a normal distribution with covariance $\nu^2 \hat I_2$ and mean $H$ lie in a ball $B_r(H)$ with radius $r = \nu\sqrt{ -2\log\alpha }$. We use this to estimate the power 
\begin{equation}\label{beta2d}
1-\beta = \frac{1}{2\pi \nu^2}\int_{\Vert\vx\Vert> r} e^{-\frac{1}{2\nu^2}\Vert\vx - H\Vert^2}\mathrm{d}\vx.
\end{equation}
Note that \eqref{beta2d} holds only when $\hat C=\nu^2\hat I_2$. Otherwise the general formula in Theorem~\ref{thm4.1} should be used to compute the power. The predicted value of the power $1-\beta = 0.99$ matches well with the observed power $1-\beta_o = 0.98$ calculated using the bottom right histogram in figure \ref{fig12}. 

We can also use this idea to illustrate the spread of $F$ around $H$, see figure \ref{fig14}. For example, setting $\al=0.05$, our results imply that when there is an edge present, the vector $F$ (the estimated edge) lies within $B_{c_\al}(H)$ (the interior of the orange circle), where $c_\al = \nu\sqrt{ -2\log0.05 }$, with $95\%$ probability. The percentage of the computed observations (the ‘x’s) which are in $B_{c_\al}(H)$ turned out to be $95\%$ to two significant figures, thus validating our conclusion. For the actual uncertainty quantification in an experiment, we use the same circle, but its center is at $F$ rather than at $H$.
\begin{figure}[!h]
\centering
\begin{subfigure}{0.4\textwidth}
\includegraphics[width=0.9\linewidth, height=5cm, keepaspectratio]{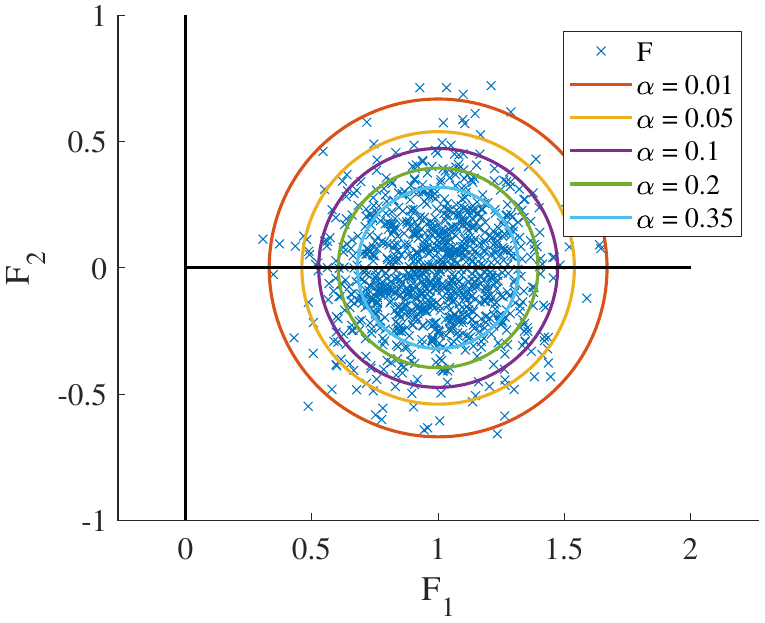}
\subcaption{CR plots} \label{14a}
\end{subfigure}
\begin{subfigure}{0.4\textwidth}
\includegraphics[width=0.9\linewidth, height=5cm, keepaspectratio]{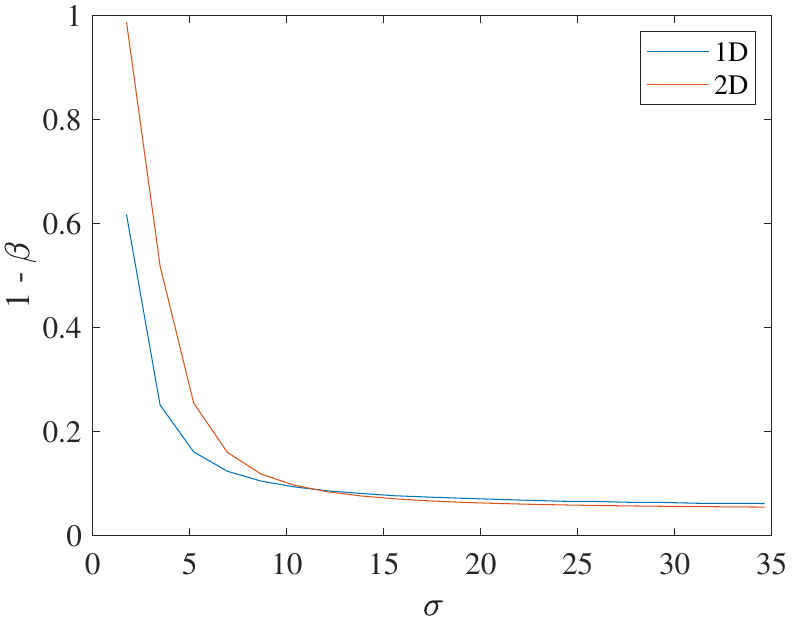}
\subcaption{Comparison of the $1 - \bt$ vs $\sigma$ curves in 1D and 2D} \label{14b}
\end{subfigure}
\caption{(a) - Computed observations, $F$ ($10^3$ samples), when there is an edge present. The individual observations are marked by ‘x’. The observations are scaled by $1/|H|$ to reflect the true jump magnitude, i.e., 1 in this case. We also show the level sets, $B_{c_\al}(H)$, of the predicted PDF for $F$ for a variety of $\alpha$. These sets illustrate the spread of $F$ around $H$.  (b) - plot of $1 - \beta$ vs $\sigma$ comparing the 1D (with $u$ linear) and 2D approaches.}
\label{fig14}
\end{figure}

%This means that when there is no edge present, we will identify correctly that there is no edge $95\%$ of the time, and when there is an edge present this will be identified correctly $99\%$ of the time. 

Interestingly, the power and AUC of the 2D method, $1-\beta = 0.99$ and $\text{AUC}= 0.99$, are significantly higher than in the equivalent 1D example when $u$ is linear ($1 - \beta = 0.64$ and $\text{AUC} = 0.92$), and thus the 2D approach is preferred for the purposes of edge detection given also that the edge direction is not required. In figure \ref{14b}, we plot $\bt$ vs $\sigma$ keeping the jump size $\Delta f = 1$ and $\alpha = 0.05$ fixed and compare the 1D and 2D methods. The 2D approach offers greater or equal power to the 1D method across all $\sigma$, and when $\sigma$ is sufficiently large $(1 - \beta) \to \alpha$. The reason for the higher power using the 2D method is easy to understand. The edge response function $f_T(\chx)$ is constant along the edge, while the GRF oscillates. This gives the test statistic a better chance to reduce the noise in the reconstruction by averaging it, which has no effect on the deterministic part of the signal.

\subsection{Uncertainty in estimating the edge direction}\label{ssec:edm}

In addition to the uncertainty quantification of estimating the vector $H$, one can be interested to estimate the uncertainty in the direction, $H/|H|$, and magnitude, $|H|$, of the edge as separate quantities. Similarly to Section \ref{ssec:uq}, by inverting the appropriate statistical test it is possible to obtain separate confidence regions for the edge direction and magnitude. However, the corresponding test statistics are fairly complicated and using them to obtain the desired confidence regions would take us far away from the main topic of the paper. Therefore we instead illustrate the uncertainty in these quantities by plotting their PDFs.

Let $\Omega\subset S^1$ be a small subset of the unit sphere. Recall that $F\sim f(u;H,\hat C)$ (see \eqref{main pdf} and \eqref{hypo 2D}) and $\hat C=\nu^2 \hat I_2$. Therefore we replace $\hat C$ with $\nu$ in the arguments of $f$. Clearly,
\be
\Pb(F/|F|\in \Omega)=\int_\Omega \ioi f(t\vec\Theta;H,\nu)t\dd t\dd\theta,\ 
\vec\Theta=(\cos\theta,\sin\theta).
\ee
This implies that the PDF of the unit vector $F/|F|$, defined on $S^1$ and denoted $\varphi(\theta)$, is given by 
\be
\varphi(\theta)=\ioi f(t\vec\Theta;H,\nu)t\dd t,\ \vec\Theta = (\cos\theta,\sin\theta).
\ee
Clearly, the shape of the PDF depends on $|H|$ and $\nu$, but is independent of the orientation of $H$. We plot this PDF, $\varphi(\theta)$, in the form of a polar graph (a scattering diagram) in Figure~\ref{15a} for different values of $\sigma$. All the other parameters are fixed. Recall that $\nu$ depends on $\sigma$, see \eqref{Cov main}, \eqref{vec covar}, and remark~\ref{rem:cov}. The $\sigma = \sqrt{3} = 1.7$ curve, which corresponds to the worked example covered in Section~\ref{ssec:2dexp}, is weighted strongly towards $\vec\Theta_0 = (1,0)$, which is the true direction of the edge, and we are more confident that the edge has direction near $\vec\Theta_0$. As $\sigma$ increases, the $\varphi$ curve becomes more uniform about the origin, and, e.g., when $\sigma = 173$, $\varphi$ is nearly a circle (the edge is equally likely in all directions). 

Further, given some level $0<\alpha<1$, we can find $\omega$ such that 
\be
\phi(\omega) = \int_{|\theta-\theta_0|\le\omega} \varphi(\theta)\dd \theta = 1-\alpha,
\ee
where $\vec\Theta_0=H/|H|$. Then the estimated direction of the true edge would lie within $\pm\omega$ from $\theta_0$ with $100(1-\alpha)\%$ probability. See figure \ref{15b}, where we plot $\phi(\omega)$ for $\sigma=\sqrt3$. This plot illustrates how likely it is for the estimated edge direction $F/|F|$ to deviate from the true direction $H/|H|$ by no more than a given angle.
 
For example, setting $\al = 0.05$ gives $\omega = 26^{\circ}$. Thus the direction of the edge does not deviate by more than $26^{\circ}$ from $\vec\Theta_0$ with $95\%$ probability. Using the $10^4$ samples we generated to calculate the $H + G$ histogram in the middle bottom panel of figure \ref{fig12}, we verify that approximately $95\%$ have direction within $\pm20^{\circ}$ from $\theta_0$.

%In this case, the value $\omega = 0.95 \times 90^\circ = 85.5^{\circ}$ satisfies $\phi(\omega) = 0.95$, and we have little confidence in the direction of the edge.

\begin{figure}[!h]
\centering
\begin{subfigure}{0.4\textwidth}
\includegraphics[width=0.9\linewidth, height=5cm, keepaspectratio]{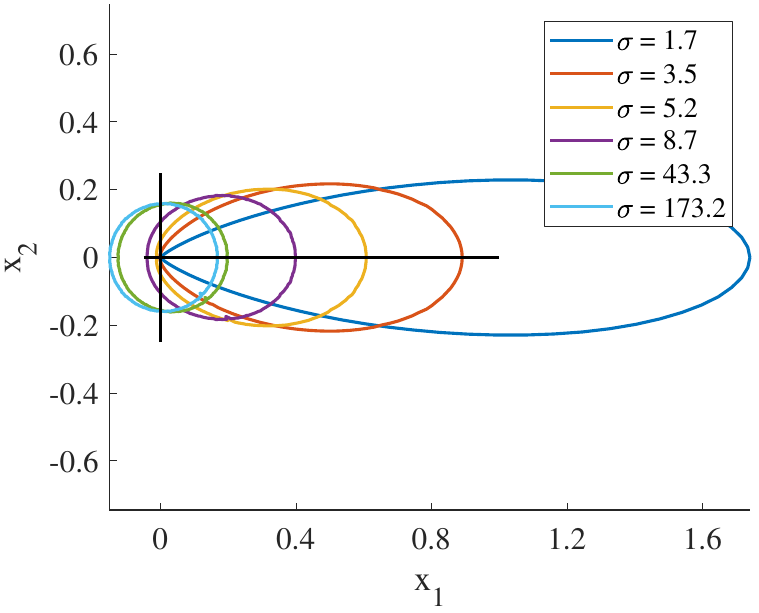}
\subcaption{$\varphi(\theta)$ polar plots} \label{15a}
\end{subfigure}
\begin{subfigure}{0.4\textwidth}
\includegraphics[width=0.9\linewidth, height=5cm, keepaspectratio]{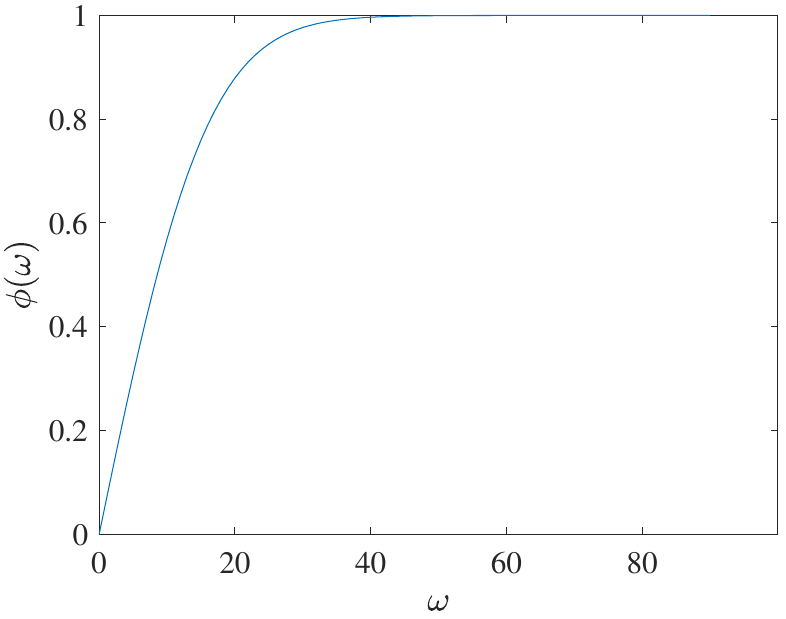}
\subcaption{$\phi(\omega)$} \label{15b}
\end{subfigure}
\caption{(a) Polar graphs of $\varphi(\theta)$ for varying $\sigma$ illustrating the probability of the edge occuring in a given direction. (b) Plot of $\phi(\omega)$ for $\sigma = \sqrt{3}$. }
\label{fig15}
\end{figure}

\subsection{Uncertainty in estimating the edge magnitude}\label{ssec:magn}
In a similar vein to the previous section, here we use our derived PDF for $F$ to quantify the uncertainty of the edge magnitude. It is clear that 
\be
\varphi(t) = t\int_0^{2\pi}f(t\vec\Theta)\dd\theta,\ t\ge0,
\ee
is the PDF for the estimated edge magnitude, $|F|$. We plot a set of $\varphi$ curves for varying noise levels $\sigma$ in figure \ref{16a}. As the noise level decreases, the PDF tends to a $\delta$ function centered on 1, the true edge magnitude. Conversely, when the noise increases, the Gaussian has greater standard deviation and there is less certainty of the edge magnitude.

Furthermore, $\phi(r)=\int_{|t-t_0|\le r}\varphi(t)\dd t$, $r<t_0$, is the probability that the estimated magnitude does not deviate by more than $r$ from the true magnitude $t_0=|H|$. For any $0<\al<1$, we can also calculate the $r$ such that $\phi(r) = 1-\al$ to derive an interval centered at $|H|$ that contains $|F|$ with $100(1-\al)\%$ probability. For example, $\phi(0.43) = 0.95$ meaning that $0.57 < |F| < 1.43$ with $95\%$ probability. See figure \ref{16b} for a plot of $\phi(r)$ for $\sigma=\sqrt 3$. 

\begin{figure}[!h]
\centering
\begin{subfigure}{0.4\textwidth}
\includegraphics[width=0.9\linewidth, height=5cm, keepaspectratio]{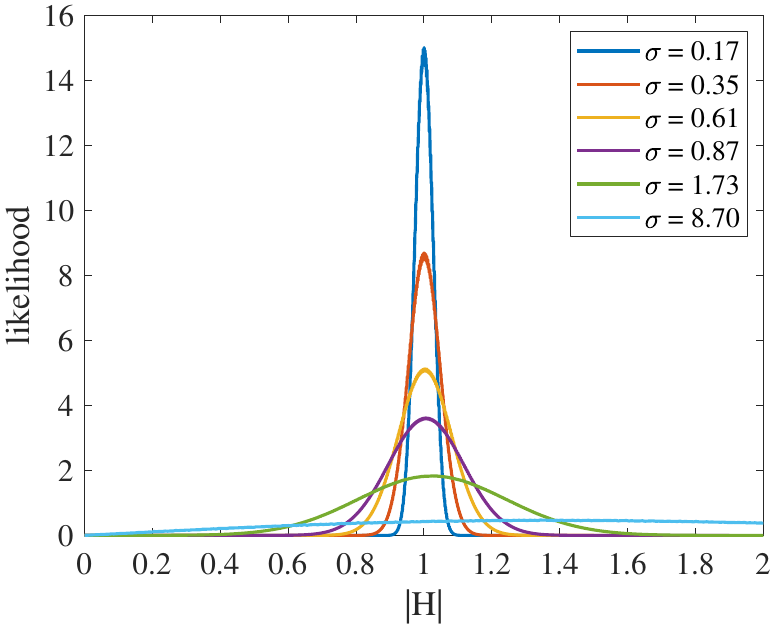}
\subcaption{$\varphi(t)$ plots (PDFs for $|H|$)} \label{16a}
\end{subfigure}
\begin{subfigure}{0.4\textwidth}
\includegraphics[width=0.9\linewidth, height=5cm, keepaspectratio]{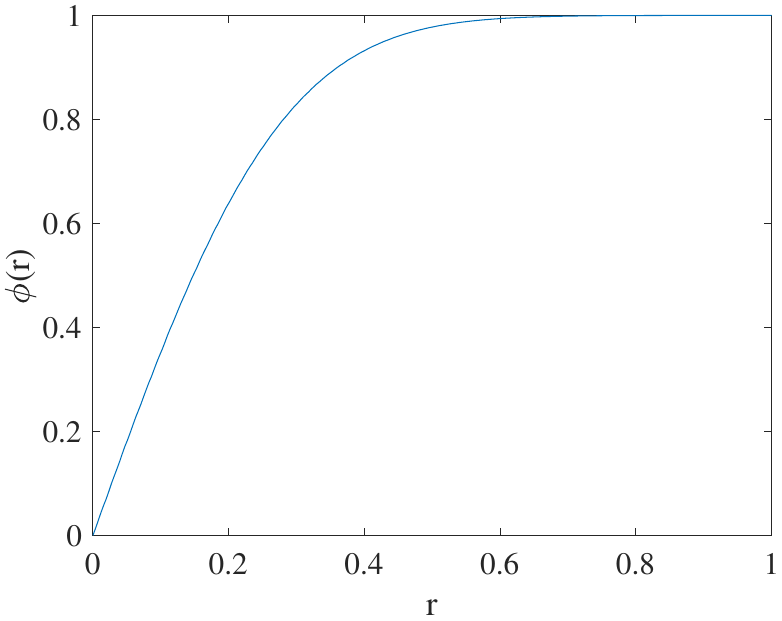}
\subcaption{$\phi(r)$} \label{16b}
\end{subfigure}
\caption{(a) Polar graphs of $\varphi(t)$, the PDF for $F$, for varying $\sigma$. (b) Plot of $\phi(r)$.}
\label{fig16}
\end{figure}

\subsection{Edge detection on the macro scale}\label{ssec:macro}
In the previous examples, we quantified in terms of $\beta$ the probability that an edge is visible microlocally in the reconstruction given $\e$, $\sigma^2$ and $\Delta f$. Throughout this section thus far, $\vx_0$ has remained fixed and we generated a set of reconstructions with different random noise draws to estimate $F$. In this section, we illustrate how we can employ $F$ to detect edges on the full image (macro) scale. In this example, a globally reconstructed image is fixed, and we slide the window $\vx_0+\e B_\rho(\bold 0)$ across the image by varying $\vx_0$ to estimate the likelihood of an edge occuring at $\vx_0$. We show the result of this in figure \ref{fig17} with $\sigma = 4 \sqrt{3}$. This equates to the ratio of the $L^2$ norms of the noise and the nonrandom signal in the CT data being $15\%$ (i.e., NSR$=15\%$). All other parameters, e.g., $\e$ are kept the same as before. We see in figure \ref{17b} that $|F|$ highlights the true edges well and the edge map can be recovered accurately as in figure \ref{17c} using a simple threshold. We show $|F|$ with no noise in figure \ref{17bb} to confirm that the sensitivity of the test is independent of the edge direction (cf. \eqref{H no noise}). 
%To quantify this effect, if we consider the values in figure \ref{17b} as \tred{relative probabilities} of an edge, then we calculate the AUC score $\text{AUC} = 0.996$ comparing to the true edge locations, which indicates excellent edge detection performance. \tred{I suggest either to delete the preceding sentence or, at least, rephrase it. Neighboring patches are correlated, so this probability estimate is quite inaccurate.} 

The directions of the edges estimated from noisy data are represented in figure \ref{17d} using the angle $\theta_F$, where the direction of the edge is given by $(\cos\theta_F,\sin\theta_F)$. We only show the $\theta_F$ values at the edge locations calculated in figure \ref{17b} and set all the other values to zero. Additionally, we highlight some of the estimated edge directions $F/|F|$ by the white arrows in figure~\ref{17c}. This is an approximation to the classical wavefront set elements of a disk.

In the macro setting, the hypothesis testing theory of Section~\ref{ssec:2dexp} no longer provides the statistical guarantee stated in \eqref{level} as the neighboring windows are not independent. This is the common challenge of all scan statistics. The proper application of a pointwise hypothesis test in a scanning regime is a separate direction of research in statistics \cite{gk24}. Working out this issue is well beyond the scope of this work. We show the results of figure \ref{fig17} as an initial test of how our method performs for full-scale image edge detection, and the results appear promising. 

Additionally, we woud like to underscore that to apply our method, we require the reconstructed image (or, at least, the region of interest inside the image) to be computed on a grid with step size, which is a fraction of $\e$. This is necessary in order to calculate the integrals like those in \eqref{5.1}. We believe that reconstructing an image with grid step size $\gtrsim\e$, as is currently the common practice, may lead to some information loss. In future work, we aim to investigate further whether such small step size methods can lead to better edge detection performance or bring about other benefits. 

The preceding comment applies if Assumption~\ref{noi} holds. If \eqref{flex moments} holds with some $\vartheta$ such that $\vartheta(\e)\to\infty$ as $\e\to0$, then it is acceptable to reconstruct $f^{\text{rec}}_{\e,\eta}(\vx)$ on a grid with the original step size $\e$, because $\e\ll\e^\prime$, where $\e^\prime$ is the appropriate resolution for the given noise strength (see the discussion at the end of Section~\ref{sec:setting_mainres}).

\begin{figure}
\centering
\begin{subfigure}{0.24\textwidth}
\includegraphics[width=0.9\linewidth, height=3.2cm, keepaspectratio]{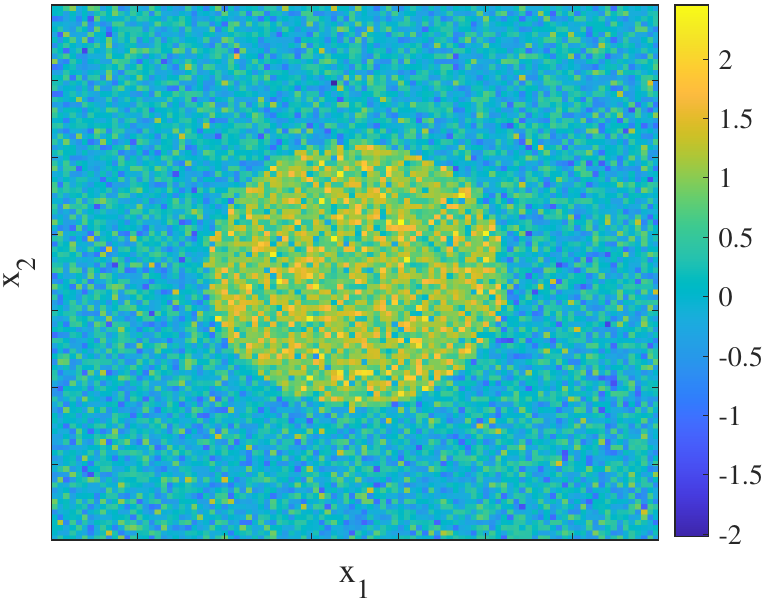}
\subcaption{reconstruction} \label{17a}
\end{subfigure}
\begin{subfigure}{0.24\textwidth}
\includegraphics[width=0.9\linewidth, height=3.2cm, keepaspectratio]{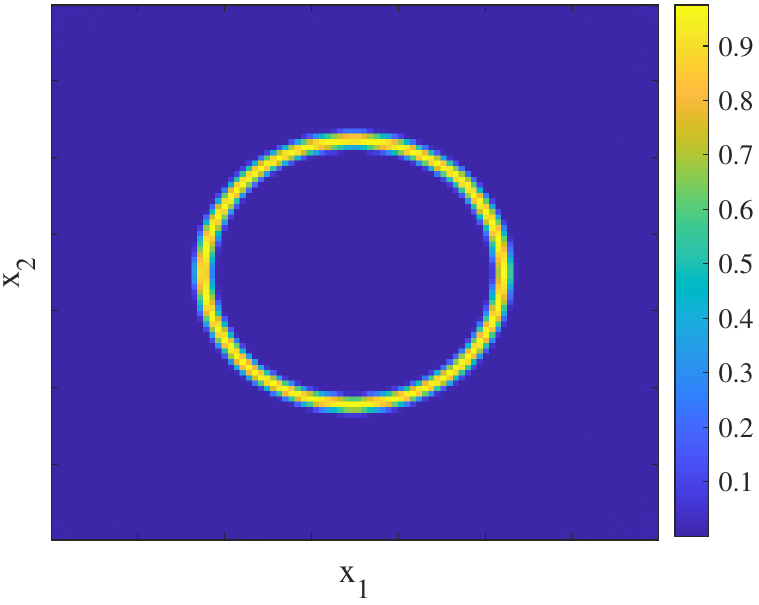}
\subcaption{$|F|$ (no noise)} \label{17bb}
\end{subfigure}
\begin{subfigure}{0.24\textwidth}
\includegraphics[width=0.9\linewidth, height=3.2cm, keepaspectratio]{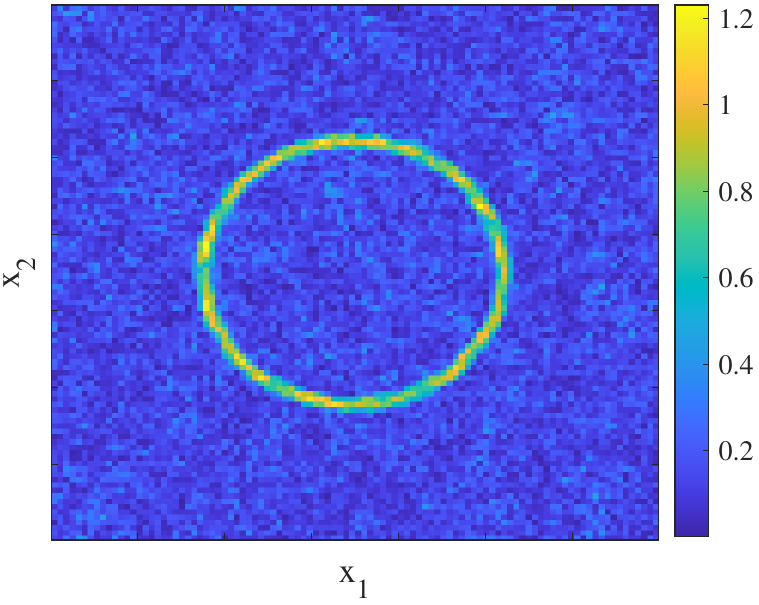}
\subcaption{$|F|$ (with noise)} \label{17b}
\end{subfigure}
\\
\begin{subfigure}{0.24\textwidth}
\includegraphics[width=0.9\linewidth, height=3.2cm, keepaspectratio]{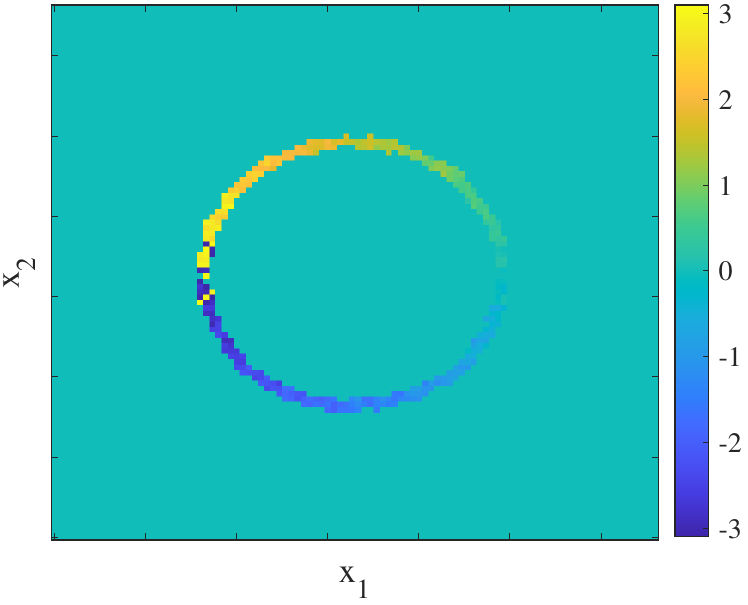}
\subcaption{$\theta_F$} \label{17d}
\end{subfigure}
\begin{subfigure}{0.24\textwidth}
\includegraphics[width=0.9\linewidth, height=3.2cm, keepaspectratio]{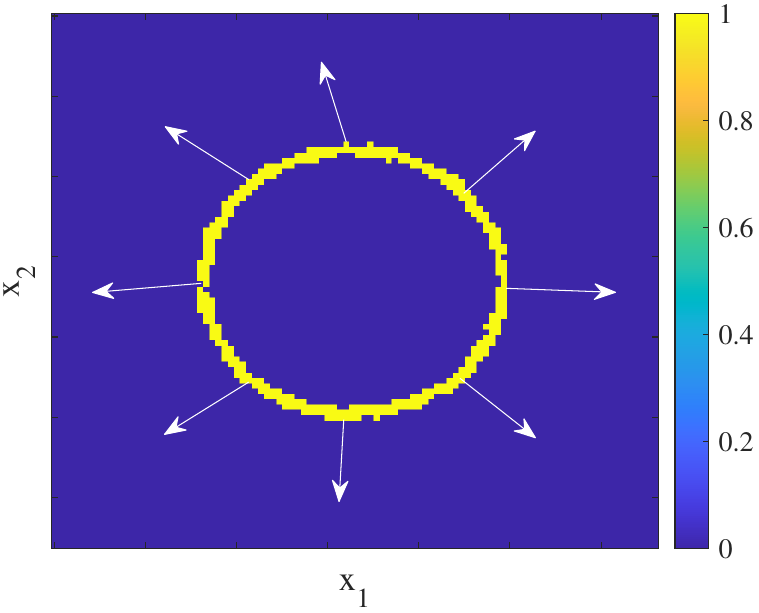}
\subcaption{$\text{WF}(f)$ approximation} \label{17c}
\end{subfigure}
\caption{(a) Noisy reconstruction of image phantom on full image scale. (b) and (c) Map of $|F|$ without and with noise, respectively. (d) Map of $\theta_F$, where $F = |F|(\cos\theta_F, \sin\theta_F)$, but set to zero when there is no edge. (e) extracted edges with directions. }
\label{fig17}
\end{figure}

\section{Appendix}

\begin{proposition}\label{prop2}
Let $D\subset\br^n$ be a bounded domain. Let $u(\chx)$ be any $L^1(D)$ deterministic function and $N^{\text{rec}}(\chx)$ be the GRF (a random function) constructed in \cite{AKW2024_1}, see also Theorem~\ref{GRF_thm}. Consider the random variable
\begin{align}\label{grv}
U=\int\limits_{D}u(\chx) N^{\text{rec}}(\chx) \dd \chx.
\end{align}
Then, $U$ is a Gaussian random variable with mean $0$ and variance $\gamma^2=\int\limits_{\mathbb{R}^{2n}}u(\chx)u(\chy)C(\chx-\chy)\dd\chx\dd\chy$.
\end{proposition}

\begin{proof}
%Introduce an abstract probability space $\{\Omega, B(\Omega), P\}$ and think of the random function $N^{\text{rec}}(\chx):=N^{\text{rec}}(\chx,\omega)$ as a measurable function on $D\times \Omega$. See e.g. \cite[Chapter 1]{rozanov_book} and 
Let $C:=C(D,\br)$ be the collection of all continuous functions $f:D\to\br$ metrized by
\be\label{Cmetric}
d(f,g)=\sup_{\chx\in D}|f(\chx)-g(\chx)|,\ f,g\in C.
\ee
It is shown in \cite{AKW2024_1} that $N^{\text{rec}}(\chx)$ is a $C$-valued random variable on an abstract probability space $\{\Omega, B(\Omega), P\}$ (see \cite[Chapter 6, Section 1.1]{khos_book_02}). In particular, for any given $\omega\in\Omega$, $N^{\text{rec}}(\chx,\omega):D\to\br$ is a continuous function. 
%In particular this implies the continuity in the mean for the GRF, i.e. $\lim\limits_{x\to y}\Eb[N^{\text{rec}}(\chx)-N^{\text{rec}}(\chy)]^2=0,$ see e.g. \cite [Example 1, Page 15]{rozanov_book}.
Hence $u(\chx)$ gives rise to the random variable $U$:
\begin{align}
U &:\Omega \to \Rb,\ U(\omega):= \int\limits_{D} u(\chx) N^{\text{rec}}(\chx,\omega) \dd \chx.
\end{align}
Indeed, since $u(\chx)\in L^1(D)$, the map $C\to\mathcal B(\br)$ given by $f\to\int_D f(\chx)u(\chx)\dd\chx$ is clearly continuous. Here $\mathcal B(\br)$ is $\br$ endowed with the standard topology. This map is therefore measurable, and \cite[Chapter 6, Theorem 1.1.1]{khos_book_02} implies that $U(\omega)$ is a measurable function on $\Omega$, i.e. $U$ is a random variable.

Consider a sequence of partitions of $D=\cup_{i=1}^n D_i$ such that $D_i\cap D_j=\varnothing$, $i\neq j$, and $\text{diam}(D_i)=O(n^{-1/2})$. The dependence of $D_i$ on $n$ is omitted. Introduce a sequence of random variables:
\begin{align}\label{2dGauss}
U_{n}= \sum\limits_{i=1}^{n} \bar u_iN^{\text{rec}}(\chx_i;\omega),\ \bar u_i:=\int_{D_i}u(\chx)\dd \chx.
\end{align}  
where $\chx_i\in D_i$ is an arbitrary point. 

Clearly, $U_n\sim \mathcal N(0,\gamma_n^2)$, where $\gamma_n^2=\sum_{i,j=1}^n \bar u_i \bar u_j C(\chx_i-\chx_j)$ and $C(\chx)$ is defined in \eqref{Cov main}. It remains to show $\Eb(U -U_{n})^2\to 0$. We have
\be\label{conv stp 1}\bs
\Eb(U-U_n)^2=&\Eb\bigg(\sum_{i=1}^{n} \int_{D_i}u(\chx)\big(N(\chx;\omega)-N^{\text{rec}}(\chx_i;\omega)\big)\dd \chx\bigg)^2\\
=&\sum_{i,j=1}^{n} \int_{D_i}\int_{D_j}u(\chx)u(\chy)\Eb\big(N(\chx;\omega)-N^{\text{rec}}(\chx_i;\omega)\big)\big(N(\chy;\omega)-N^{\text{rec}}(\chx_j;\omega)\big)\dd \chx\dd \chy\\
=&\sum_{i,j=1}^{n} \int_{D_i}\int_{D_j}u(\chx)u(\chy)\big(C(\chx,\chy)-C(\chx,\chx_j)-C(\chx_i,\chy)+C(\chx_i,\chx_j)\big)\dd \chx\dd \chy.
\end{split}
\ee
By construction, $|\chy-\chx_j|=O(n^{-1/2})$. By the continuity of $C$, we get from \eqref{conv stp 1}:
\be\label{conv stp 2}\bs
\Eb(U-U_n)^2\le &O(n^{-1/2})\sum_{i,j=1}^{n} \int_{D_i}\int_{D_j}|u(\chx)||u(\chy)|\dd \chx\dd \chy\\
=&O(n^{-1/2})\bigg[\int_D |u(\chx)|\dd \chx\bigg]^2\to0,\ n\to\infty.
\end{split}
\ee

This shows that $U_{n}\to U$ in $L^2$. Lemma~\ref{limit} stated below implies $U\sim\mathcal N(0,\gamma^2)$, where $\gamma^2=\lim_{n\to\infty} \gamma_n^2$. See \cite[Page 17]{rozanov_book} for a similar proof. 

Finally, substituting $u(\chy)=y_j$ gives the result for $G_j$, $j=1,2$. 

\begin{lemma}[Example 3.1.12 and Lemma 3.1.9 of \cite{sokol2013advanced}]
\label{limit}
Let $(\xi_n)$ and $(\gamma_n^2)$ be real sequences with limits $\xi$ and $\gamma^2$, respectively. Then the distribution $\mathcal N(\xi_n,\gamma_n^2)$ converges weakly to $\mathcal N(\xi,\gamma^2)$ as $n\to \infty$.
\end{lemma}
\end{proof}

\section*{Acknowledgments}
JWW wishes to acknowledge funding support from The Cleveland Clinic Foundation, The Honorable Tina Brozman Foundation, the V Foundation, and the National Cancer Institute R03CA283252-01. AK is thankful to Prof. Eugene Katsevich, Statistics Department, University of Pennsylvania, for helpful discussions about confidence regions.

\bibliographystyle{abbrv}
\bibliography{refs, My_Collection}
\end{document}